\documentclass[12pt]{article}
\usepackage{latexsym,amsmath,amssymb,amsfonts,amsthm,bbm}
\usepackage{natbib}
\usepackage{enumerate}
\RequirePackage[colorlinks,citecolor=blue,urlcolor=blue,breaklinks]{hyperref}
\usepackage{graphicx}
\usepackage{graphics,enumerate}
\usepackage[x11names,svgnames]{xcolor}

\newtheorem{theorem}{Theorem}
\newtheorem{proposition}[theorem]{Proposition}
\newtheorem{lemma}[theorem]{Lemma}

\newtheorem{defn}{Definition}

\DeclareMathOperator*{\argmax}{argmax}

\newcommand{\E}{{\mathbb E}}

\newcommand{\R}{{\mathbb R}}
\renewcommand{\P}{{\mathbb P}}

\newcommand{\G}{{\mathcal{G}}}

\newcommand{\F}{{\cal F}}

\newcommand{\scs}{{\mathcal{S}}}

\newcommand {\one}{{\mathbbm{1}}}

\newcommand{\hel}{d_\mathrm{H}^2} 
\newcommand{\shel}{d_\mathrm{H}} 
\newcommand{\kl}{d_\mathrm{KL}^2}

\newcommand{\kld}{d_\mathrm{X}^2}

\newcounter{rcnt}[section]

\def\qt#1{\qquad\text{#1}}

\begin{document}

\title{Adaptation in log-concave density estimation}

\author{Arlene K. H. Kim$^\ast$, Adityanand Guntuboyina$^\dag$ and Richard J. Samworth$^\ddag$ \\ 
 $^{\ast,\ddag}$ University of Cambridge\\ and\\ $^\dag$ University of
 California Berkeley}  
\footnotetext[2]{Research supported by NSF Grant DMS-1309356.}
\footnotetext[3]{Research supported by an EPSRC Early Career
  Fellowship and a grant from the Leverhulme Trust.} 

\maketitle

\begin{abstract}
The log-concave maximum likelihood estimator of a density on the real line based on a sample of size $n$ is known to attain the minimax optimal rate of convergence of $O(n^{-4/5})$ with respect to, e.g., squared Hellinger distance.  In this paper, we show that it also enjoys attractive adaptation properties, in the sense that it achieves a faster rate of convergence when the logarithm of the true density is $k$-affine (i.e.\ made up of $k$ affine pieces), provided $k$ is not too large.  Our results use two different techniques: the first relies on a new Marshall's inequality for log-concave density estimation, and reveals that when the true density is close to log-linear on its support, the log-concave maximum likelihood estimator can achieve the parametric rate of convergence in total variation distance.  Our second approach depends on local bracketing entropy methods, and allows us to prove a sharp oracle inequality, which implies in particular that the rate of convergence with respect to various global loss functions, including Kullback--Leibler divergence, is $O\bigl(\frac{k}{n}\log^{5/4} n\bigr)$ when the true density is log-concave and its logarithm is close to $k$-affine.  
\end{abstract}

\section{Introduction}

It is well known that nonparametric shape constraints such as monotonicity, convexity or log-concavity have the potential to offer the practitioner the best of both the nonparametric and parametric worlds: on the one hand, the infinite-dimensional classes allow considerable modelling flexibility, while on the other one can often obtain estimation procedures that do not require the choice of tuning parameters.  Examples include isotonic regression \citep{VanEeden1956,BBBB1972}, convex regression \citep{Hildreth1954,SeijoSen2011,LimGlynn2012}, generalised additive models \citep{ChenSamworth2016}, the Grenander estimator \citep{Grenander1956}, convex density estimation \citep{GJW2001}, independent component analysis \citep{SamworthYuan2012} and many others.  See \citet{GroeneboomJongbloed2014} for a recent book-length treatment of the field. 

These attractive properties have led to intensive efforts in recent years, to try to understand the theoretical properties of shape-constrained estimators.  In some cases, for instance, it is now known that these estimators can achieve minimax optimal rates of convergence; see, for example, \citet{Birge1987} for the Grenander estimator, \citet{BaraudBirge2016} for $\rho$-estimators, \citet{KimSamworth2016} for the log-concave maximum likelihood estimator and \citet{HanWellner2016} for convex regression estimators.  However, the fact that these estimators are tuning-free raises the prospect of an additional allure, namely that they might adapt to certain types of data generating mechanisms in the sense of attaining a faster rate of convergence than that predicted by the `worst-case' minimax theory. 

The purpose of this paper is to explore this adaptation phenomenon in the context of log-concave density estimation.  Recall that a density $f$ on the real line is said to be log-concave if it is of the form $\exp(\phi)$ for some concave function $\phi: \R \rightarrow [-\infty,\infty)$.  We write $\mathcal{F}$ for the set of all upper semi-continuous log-concave densities.  The class $\mathcal{F}$ serves as a particularly attractive nonparametric surrogate for the class of Gaussian densities, because it is closed under linear transformations, marginalisation, conditioning and convolution, and because it contains many commonly encountered parametric families of unimodal densities with exponentially decaying tails.  For this reason, the log-concave maximum likelihood estimator of $f$, first introduced by \citet{Walther2002}, has been studied in great detail in recent years; see, for example, \citet{PWM2007,DumbgenRufibach2009,SereginWellner2010,CuleSamworth2010,SchuhmacherDumbgen2010,CSS2010,DSS2011}.  

Very recently, \citet{KimSamworth2016} proved that if $X_1,\ldots,X_n$ are an independent sample from $f_0 \in \mathcal{F}$, then\footnote{Here, we write $a_n \asymp b_n$ to mean that $0 < \liminf_{n \rightarrow \infty} |a_n/b_n| \leq \limsup_{n \rightarrow \infty} |a_n/b_n| < \infty$.}
\begin{equation}
\label{Eq:Minimax}
\inf_{\tilde{f}_n} \sup_{f_0 \in \mathcal{F}} \mathbb{E}_{f_0}{\hel}(\tilde{f}_n,f_0) \asymp n^{-4/5},
\end{equation}
and that the log-concave maximum likelihood estimator $\hat{f}_n$ based on $X_1,\ldots,X_n$ attains this minimax optimal rate.  Here, the infimum is taken over all estimators $\tilde{f}_n$ of $f_0$, and $\hel(f,g) := \int_{-\infty}^\infty (f^{1/2} - g^{1/2})^2$ denotes the squared Hellinger distance.  In fact, there are various other choices of global loss function that one can study, including the total variation distance and Kullback--Leibler divergence, defined respectively by
\[
d_{\mathrm{TV}}(f,g) := \frac{1}{2}\int_{-\infty}^\infty |f - g|, \quad \text{and} \quad d_{\mathrm{KL}}^2(f,g) := \int_{-\infty}^\infty f \log \frac{f}{g},
\]
where we set $d_{\mathrm{KL}}^2(f,g) := \infty$ if the support of $f$ is not contained in the support of $g$.  We recall the standard inequalities $d_{\mathrm{TV}}^2(f,g) \leq \hel(f,g) \leq d_{\mathrm{KL}}^2(f,g)$ that relate these loss functions.  In fact, in this work, we will also be interested in another notion of divergence: by an application of Remark~2.3 of \citet{DSS2011} to the function $x \mapsto \log \frac{f_0(x)}{\hat{f}_n(x)}$, we have
that  
\[
d_{\mathrm{KL}}^2(\hat{f}_n,f_0) \leq \frac{1}{n}\sum_{i=1}^n \log \frac{\hat{f}_n(X_i)}{f_0(X_i)} =: d_X^2(\hat{f}_n,f_0).
\]
Thus, an upper bound on the risk of the log-concave maximum likelihood estimator in the $d_X^2$ divergence immediately yields bounds in each of the other global loss functions mentioned above.  

The log-concave maximum likelihood estimator can be expressed as
\[
\hat{f}_n(x) = \left\{ \begin{array}{ll} \exp\bigl\{\min(b_1x-\beta_1,\ldots,b_m x - \beta_m)\bigr\} &\mbox{if $x \in [\min(X_1,\ldots,X_n),\max(X_1,\ldots,X_n)]$} \\
0 & \mbox{otherwise,} \end{array} \right.
\]
for some $m \in \mathbb{N}$ and $b_1,\ldots,b_m,\beta_1,\ldots,\beta_m \in \mathbb{R}$.   This motivates the thought that if $\log f_0$ is itself composed of a relatively small number of affine pieces (e.g.\ the logarithm of a Laplace density comprises two affine pieces), then we might expect $\hat{f}_n$ to converge to $f_0$ at an especially fast rate.  

To this end, for $k \in \mathbb{N}$, we define $\mathcal{F}^k$ to be the class of log-concave densities $f$ for which $\log f$ is \emph{$k$-affine} in the sense that there exist intervals $I_1,\ldots,I_k$ such that $f$  is supported on $I_1 \cup \ldots \cup I_k$, and $\log f$ is affine on each $I_j$.  We then study adaptation in log-concave density estimation via two different approaches.  The first, presented in Section~\ref{Sec:TV}, establishes risk bounds in total variation distance for true densities that are close to $\mathcal{F}^1$, showing in some cases (such as when the true density is uniform on a compact interval), that the log-concave maximum likelihood estimator achieves the parametric rate of convergence.  Our key tool for this approach is an analogue of Marshall's inequality \citep{Marshall1970}, which we use to relate $\sup_{x \in \mathbb{R}} |\hat{F}_n(x) - F_0(x)|$ to $\sup_{x \in \mathbb{R}} |\mathbb{F}_n(x) - F_0(x)|$, where $\mathbb{F}_n$, $F_0$ and $\hat{F}_n$ denote the empirical distribution function and the distribution functions corresponding to $f_0$ and $\hat{f}_n$ respectively.   An attraction of this strategy is that the true density need not be assumed to be log-concave.

Our second approach, developed in Section~\ref{Sec:kaffine}, studies more general adaptation of the log-concave maximum likelihood estimator to densities in $\mathcal{F}^k$ via local bracketing entropy methods.  More precisely, we provide risk bounds in the $d_X^2$ divergence when the true density is log-concave and close to $\mathcal{F}^k$, which reveal that a rate of $\frac{k}{n}\log^{5/4} n$ can be attained.  Thus, when $k$ is relatively small, we obtain a significant improvement over the minimax rate.  

There has been considerable interest in adaptation in shape-constrained estimation, especially in recent years, on problems including decreasing density estimation \citep{Birge1987}, isotonic regression \citep{Zhang2002,CGS2014}, matrix estimation under shape constraints \citep{CGS2015} and convex regression \citep{ChenWellner2016,HanWellner2016}.  However, all of these works consider the least squares estimator, which has a more explicit expression as a projection onto a convex set.  The class of log-concave densities is not convex, and the maximum likelihood estimator does not have such a simple characterisation, so we have to develop new techniques.  We finally mention the work of \citet{BaraudBirge2016}, who study a procedure called a $\rho$-estimator in various shape-constrained density estimation problems.  We discuss their results in the context of log-concave density estimation in Section~\ref{Sec:kaffine}.

Proofs of our main results are given in Sections~\ref{Sec:ProofsMain} and~\ref{Sec:kaffineproofs}.  These rely on several auxiliary results, which are presented in the Appendix.  We conclude this introduction with some notation used throughout the paper.  Given a function $g:\mathbb{R} \rightarrow \mathbb{R}$, we write $\|g\|_\infty := \sup_{x \in \mathbb{R}} |g(x)|$.  For $f,g \in \F$, we write $D_f := \{x:f(x) > 0\} = \{x:\log f(x) > -\infty\}$ for the domain of $\log f$, and write $f \ll g$ if $D_f \subseteq D_g$.  Also for $f \in \mathcal{F}$, let $\mu_f := \int_{-\infty}^\infty xf(x) \, dx$, $\sigma_f^2 := \int_{-\infty}^\infty (x-\mu_f)^2f(x) \, dx$ and $\mathcal{F}^{0,1} := \{f \in \mathcal{F}:\mu_f=0,\sigma_f^2=1\}$.  We use $C$ to denote a generic universal positive constant, whose value may be different at different instances.

\section{Rates for densities that are close to log-affine on their support}
\label{Sec:TV}

This section concerns settings where the true density is close to $\mathcal{F}^1$, the class of densities that are 
log-affine on their support, but not necessarily log-concave.  It will be convenient to have an explicit parametrisation of such densities.  Let $\mathcal{T}_0 := \{(s_1,s_2) \in \mathbb{R}^2:s_1 < s_2\}$ and 
\[
\mathcal{T} := (\mathbb{R} \times \mathcal{T}_0) \, \bigcup \, \bigl((0,\infty) \times \{-\infty\} \times \mathbb{R}\bigr) \, \bigcup \, \bigl((-\infty,0) \times \mathbb{R} \times \{\infty\}\bigr).
\]
Now, for $(\alpha,s_1,s_2) \in \mathcal{T}$, let 
\[
f_{\alpha,s_1,s_2}(x) := \left\{ \begin{array}{ll} \frac{1}{s_2-s_1}\one_{\{x \in [s_1,s_2]\}} & \mbox{if $\alpha=0$} \\
\frac{\alpha}{e^{\alpha s_2}-e^{\alpha s_1}} e^{\alpha x} \one_{\{x \in [s_1,s_2]\}} & \mbox{if $\alpha \neq 0$.} \end{array} \right.  
\]
Then we can write
\[
\mathcal{F}^1 = \{f_{\alpha,s_1,s_2}:(\alpha,s_1,s_2) \in \mathcal{T}\}.
\]
Thus the class $\mathcal{F}^1$ consists of uniform and (possibly truncated) exponential densities.  It is also convenient to define a continuous, strictly increasing function $q:\mathbb{R} \rightarrow [0,1]$ by
\begin{equation}
\label{qfunc}
q(x) := \left\{ \begin{array}{ll} \frac{x-2+e^{-x}(x+2)}{x\{1-e^{-x}(x+1)\}}
    & \mbox{for $x \neq 0$} \\ 
\frac{1}{3} & \mbox{for $x = 0$,} \end{array} \right.
\end{equation}
and to set $\rho(x) := \frac{1+q(x)}{1-q(x)}$.  As a preliminary calculation, we note that for $x \geq 2$,
\[
q(x)= 1-\frac{2}{x}+ \frac{x}{e^x - (1+x)} \leq 1-\frac{1}{x},
\]
so that $\rho(x) \leq \max\{\rho(2),\rho(x)\} \leq \max(3,2x)$ for all $x \in \mathbb{R}$.
\begin{theorem}\label{ct}
Let $f_0$ be any density on the real line, let $X_1,\ldots,X_n \stackrel{\mathrm{iid}}{\sim} f_0$ for some $n \geq 5$, and let $\hat{f}_n$ denote the corresponding log-concave maximum likelihood estimator.  Fix an arbitrary $f_{\alpha,s_1,s_2} \in \mathcal{F}^1$, write $\kappa^* := \alpha(s_2-s_1)$, let $d_{\mathrm{TV}} := d_{\mathrm{TV}}(f_{\alpha,s_1,s_2},f_0)$ and let $d_{\mathrm{KS}}^{(n)} := \|F_{\alpha,s_1,s_2}^n - F_0^n\|_\infty + \|(1-F_{\alpha,s_1,s_2})^n - (1-F_0)^n\|_\infty$, where $F_{\alpha,s_1,s_2}$ and $F_0$ are the distribution functions corresponding to $f_{\alpha,s_1,s_2}$ and $f_0$ respectively.  Then, for $t \geq 0$, the following two bounds hold:
  \begin{align}\label{ct.eq}
(i) \ &\P_{f_0}\bigl[d_{\mathrm{TV}}(\hat{f}_n, f_0) \geq t + \{1+2\rho(|\kappa^*|)\}d_{\mathrm{TV}}\bigr] \leq 2 \exp\biggl(-\frac{nt^2}{2\rho^2(|\kappa^*|)}\biggr) + d_{\mathrm{KS}}^{(n)}, \\
\label{ct.eq2}(ii) \ &\P_{f_0}\bigl\{d_{\mathrm{TV}}(\hat{f}_n, f_0) \geq t + (1+6\log n)d_{\mathrm{TV}}\bigr\} \leq 2\exp\biggl(-\frac{nt^2}{18\log^2 n}\biggr) + \frac{1}{n^{1/2}} + d_{\mathrm{KS}}^{(n)},
  \end{align}
where we interpret~\eqref{ct.eq} as uninformative if $|\kappa^*| = \infty$.  Moreover,
\begin{align}
\label{Eq:TVrisk}
\E_{f_0} d_{\mathrm{TV}}(\hat{f}_n, f_0) \leq \inf_{f_{\alpha,s_1,s_2} \in \mathcal{F}^1}\biggl\{\frac{c_n}{n^{1/2}} + (1+c_n)d_{\mathrm{TV}} + d_{\mathrm{KS}}^{(n)}\biggr\},
\end{align}
where $c_n = c_n(f_{\alpha,s_1,s_2}) := \min\{2\rho(|\kappa^*|),6\log n\}$.
\end{theorem}
To aid with the interpretation of the second part of Theorem~\ref{ct}, first consider the case where $f_0 = f_{\alpha,s_1,s_2} \in \mathcal{F}^1$, so that $d_{\mathrm{TV}} = d_{\mathrm{KS}}^{(n)} = 0$.  In that case, provided $|\kappa^*| = |\alpha|(s_2-s_1)$ is not too large, the first term in the minimum in the definition of $c_n$ guarantees that the log-concave maximum likelihood estimator attains the parametric rate of convergence.  In particular, if $f_0 \in \mathcal{F}^1$ is a uniform density on a compact interval, then we may take $\alpha = 0 = \kappa^*$, and find that
\[
\E_{f_0} d_{\mathrm{TV}}(\hat{f}_n, f_0) \leq \frac{4}{n^{1/2}}.
\]
On the other hand, if $f_0 = f_{\alpha,s_1,s_2} \in \mathcal{F}^1$ where $|\kappa^*|$ is large (e.g.\ if it is infinite) then the second term in the minimum in the definition of $c_n$ may give a better bound, and guarantees that we attain the parametric rate up to a logarithmic factor.  More generally, if $f_0$ is any density such that 
\[
\inf_{f_{\alpha,s_1,s_2} \in \mathcal{F}^1}(d_{\mathrm{TV}} + d_{\mathrm{KS}}^{(n)}) = o(n^{-2/5}\log^{-1} n),
\]
then the rate provided by~\eqref{Eq:TVrisk} is faster than that given by the worst-case minimax theory\footnote{The fact that we work with the risk in total variation distance rather than squared total variation distance is not significant.  However, it is worth recalling that Theorem~\ref{ct} does not control the (larger) Hellinger risk.}.  
In fact, there is a special class $\mathcal{F}_* \subseteq \mathcal{F}$ such that when $f_0 \in \mathcal{F}_*$ we can prove an alternative bound on the total variation distance between $\hat{f}_n$ and $f_0$ that slightly improves and simplifies the bounds provided in Theorem~\ref{ct}.  To define this class, for $f \in \mathcal{F}$, let $D_f := \{x:f(x) > 0\}$, and let 
\begin{equation}
\label{Eq:F_*}
\mathcal{F}_* := \bigl\{f \in \mathcal{F}:f(x) = e^{\gamma x}h(x) \ \text{for all $x \in D_f$, some $\gamma \in \mathbb{R}$, $h:D_f \rightarrow [0,\infty)$ concave}\bigr\}. 
\end{equation}
As examples, if $f \in \mathcal{F}$ is concave on its (necessarily bounded) support $D_f$, then $f \in \mathcal{F}_{*}$ since we can take $\gamma = 0$ and $h(x) = f(x)$ for $x \in D_f$.  Moreover, $\mathcal{F}^1 \subseteq \mathcal{F}_*$, and the family of $\Gamma(\alpha,\beta)$ densities with $\alpha \in [1,2]$, $\beta > 0$ also belongs to $\mathcal{F}_*$.  When $f_0 \in \mathcal{F}_*$, the factors of $1+2\rho(|\kappa^*|)$, $1+6 \log n$ and $1+c_n$ in~\eqref{ct.eq},~\eqref{ct.eq2} and~\eqref{Eq:TVrisk} respectively can be replaced simply with 3.  See Proposition~\ref{Prop:ct} in the Appendix for details.


The proof of Theorem \ref{ct} is crucially based on the following analogue of the classical Marshall's inequality for decreasing density estimation \citep{Marshall1970}.  
\begin{lemma} \label{miu} 
Let $n \geq 2$, let $X_1, \dots, X_n$ be real numbers that are not all equal, with empirical distribution function $\mathbb{F}_n$, and let $\hat{f}_n$ denote the corresponding log-concave maximum likelihood estimator.  Let $X_{(1)} := \min_i X_i$ and $X_{(n)} := \max_i X_i$.  Let $f_0$ be a density such that $f_0(x) = e^{\alpha_0 x}h_0(x)$ for $x \in [X_{(1)},X_{(n)}]$, where $\alpha_0 \in \mathbb{R}$ and $h_0:[X_{(1)},X_{(n)}] \rightarrow \mathbb{R}$ is concave, and let $\kappa := \alpha_0(X_{(n)} - X_{(1)})$.  Writing $F_0$ and $\hat{F}_n$ for the distribution functions corresponding to $f_0$ and $\hat{f}_n$ respectively, we have   
  \begin{equation}\label{miu.eq2}
 \|\hat{F}_n - F_0\|_\infty \leq \rho(|\kappa|)\|\mathbb{F}_n - F_0\|_\infty. 
    \end{equation}
\end{lemma}
\noindent \textbf{Remark}: \citet{DumbgenRufibach2009} found that in all of their simulations,
\begin{equation}
\label{Eq:DR}
\|\hat{F}_n - F_0\|_\infty \leq \|\mathbb{F}_n - F_0\|_\infty,
\end{equation}
provided $F_0$ has a log-concave density.  However, since $\rho(x) \rightarrow \infty$ as $x \rightarrow \infty$, it is worth noting that Lemma~\ref{miu} is in line with their observation that `\ldots one can construct counterexamples showing that~[\eqref{Eq:DR}] may be violated, even if the right-hand side is multiplied with any fixed constant $C > 1$'.

Although Lemma~\ref{miu} is stated as a deterministic result, the main case of interest is where $X_1,\ldots,X_n$ are independent and identically distributed, and we apply the result to some density $f_0 \in \mathcal{F}_*$ (not necessarily the true density).  The original Marshall's inequality applies to the integrated Grenander estimator when $F_0$ is concave; in that case, the multiplicative factor $\rho(|\kappa|)$ can be replaced with 1.  \citet{dumbgen2007marshall} proved a similar result for the integrated version of the least squares estimator of a convex density on $[0,\infty)$; there, a multiplicative constant 2 is needed.  In the special case where $f_0$ is concave on the convex hull of the data, we can take $\alpha_0 = 0 = \kappa$, and the multiplicative constant in Lemma~\ref{miu} can also be taken to be 2.

\section{Rates for densities whose logarithms are close to $k$-affine}
\label{Sec:kaffine}

In this section, we extend significantly the class of densities for which we can prove adaptation of the log-concave maximum likelihood estimator.  Recall that, for $k \in \mathbb{N}$, the class $\mathcal{F}^k$ denotes the set of log-concave densities $f \in \mathcal{F}$ for which $\log f$ is $k$-affine. The following is the main theorem of this section. 
\begin{theorem}\label{fik}
There exists a universal constant $C > 0$ such that for every $n \geq 2$ and every $f_0 \in \F$, we have 
  \begin{equation}
    \label{fik.eq}
    \E_{f_0} \kld(\hat{f}_n, f_0) \leq \inf_{k \in \mathbb{N}}
    \biggl\{\frac{Ck}{n} \log^{5/4}n + \inf_{f_k \in \F^k} \kl(f_0, f_k) \biggr\}. 
  \end{equation}
\end{theorem}
One consequence of Theorem~\ref{fik} is that when $\log f_0$ is close to $k$-affine for some $k$ in the sense that $\inf_{f_k \in \F^k} \kl(f_0, f_k) = O\bigl(\frac{k}{n} \log^{5/4}n\bigr)$, then the log-concave MLE $\hat{f}_n$  converges to $f_0$ at rate $O\bigl(\frac{k}{n} \log^{5/4} n\bigr)$, which is almost the parametric rate when $k$ is small.  In particular, provided $k = o(n^{1/5}\log^{-5/4} n)$, the rate provided by Theorem~\ref{fik} is faster than the minimax rate over all log-concave densities of $O(n^{-4/5})$ \citep{KimSamworth2016}\footnote{Although Theorem~5 of \citet{KimSamworth2016} is stated for the squared Hellinger risk, it can easily be extended to a bound for the $d_X^2$ risk by appealing to Corollary~7.5 of \citet{vandegeerbook}.}.

A result similar to \eqref{fik.eq} was recently proved by \citet[Corollary 4]{BaraudBirge2016} for their $\rho$-estimator. More precisely, they proved that there exists a universal constant $C > 0$ such that 
  \begin{equation}\label{rho}
    \E_{f_0} \hel(\hat{f}_{\rho}, f_0 )  \leq C \inf_{k \in \mathbb{N}} \biggl\{\frac{k}{n} \log^3 \frac{en}{k} + \inf_{f_k \in \F^k} \hel(f_0, f_k) \biggr\}
  \end{equation}
where $\hat{f}_{\rho}$ denotes the $\rho$-estimator based on a sample of size $n$, defined in \citet{BaraudBirge2016}. The differences between~Theorem~\ref{fik} and~\eqref{rho} are as follows: 
  \begin{enumerate}
  \item Theorem~\ref{fik} deals with the log-concave MLE while~\eqref{rho} deals with the $\rho$-estimator. While the $\rho$-estimator is very interesting and general, at the moment, we are not aware of algorithms for computing it. On the other hand, the log-concave MLE can be easily computed via active set methods for convex optimisation \citep{DumbgenRufibach2011}. 
  \item Theorem~\ref{fik} is a sharp oracle inequality in the sense that the approximation term $\inf_{f_k \in \F^k} \kl(f_0, f_k)$ in~\eqref{fik.eq} has leading constant 1.
  \item Theorem~\ref{fik} bounds the risk with respect to the loss function $\kld$, which is larger than the squared Hellinger risk studied in~\eqref{rho}.  On the other hand, the right-hand side of~\eqref{rho} involves $\inf_{f_k \in \F^k} \hel(f_0, f_k)$, which may be smaller than the term $\inf_{f_k \in \F^k} \kl(f_0, f_k)$ that appears on the right-hand side of~\eqref{fik.eq}.
  \item Inequality~\eqref{fik.eq} has a $\log^{5/4} n$ term on the right-hand side while inequality~\eqref{rho} has a $\log^3(en/k)$ term. In the regime where $k = o(n^{1/5}\log^{-5/4} n)$ (which as explained above is really the only interesting regime for Theorem \ref{fik}), inequality~\eqref{fik.eq} has a smaller logarithmic term compared with~\eqref{rho}. 
  \end{enumerate}
Our proof of Theorem~\ref{fik} proceeds by first studying the special case where the infimum in the right-hand side of \eqref{fik.eq} is replaced by $k = 1$.  That case can be handled using empirical process theory techniques \citep[e.g.][]{vandegeerbook} together with a local bracketing entropy result for log-concave densities (cf.~Theorem~\ref{ment} below).  Before stating this result, we first recall the following definition of bracketing entropy:
\begin{defn}\label{bra}
Let $\G$ be a class of non-negative functions defined on $S \subseteq \R$. For $\epsilon > 0$, let $N_{[]}(\epsilon, \G, \shel,S)$ denote the smallest $M \in \mathbb{N}$ for which there exist pairs of functions $\{[g_{L, j}, g_{U, j}]: j = 1, \dots, M\}$ such that     
  \begin{equation*}
    \int_{-\infty}^\infty (g_{U, j}^{1/2} - g_{L, j}^{1/2})^2 \leq \epsilon^2 \qt{for every $j=1,\ldots,M$} 
  \end{equation*}
and such that for every $g \in \G$, there exists $j^* \in \{1,\ldots,M\}$ with $g_{L, j^*}(x) \leq g(x) \leq g_{U, j^*}(x)$ for every $x \in S$.  We also write $H_{[]}(\epsilon, \G, \shel,S) := \log N_{[]}(\epsilon, \G, \shel,S)$ and, when $S = \mathbb{R}$, $H_{[]}(\epsilon, \G, \shel) := H_{[]}(\epsilon, \G, \shel,\R)$, and refer to $H_{[]}(\epsilon, \G, \shel)$ as the $\epsilon$-bracketing entropy number of $\G$ under the Hellinger distance.  
\end{defn}
For $f_0 \in \mathcal{F}$ and $\delta > 0$, we also define $\F(f_0, \delta)  := \bigl\{f \in \F: f \ll f_0, d_{\mathrm{H}}(f, f_0) \leq \delta \bigr\}$.  We are now in a position to state our main local bracketing entropy bound for log-concave densities:  
\begin{theorem}\label{ment}
There exist universal constants $C,\kappa > 0$ such that for every $f_0 \in \F$ with $\upsilon := \inf \{\shel(f_0, f_1) : f_1 \in \F^1, f_0 \ll f_1\}$, and every $\epsilon > 0$,  
  \begin{equation}
    \label{ment.eq}
H_{[]}(2^{1/2}\epsilon, \F(f_0, \delta), \shel) \leq C \log^{5/4} \Bigl(\frac{1}{\delta}\Bigr)\biggl(\frac{\delta + \upsilon}{\epsilon}\biggr)^{1/2} 
  \end{equation}
provided $\delta + \upsilon < \kappa$.    
\end{theorem}
It is instructive to compare Theorem~\ref{ment} with other recent global bracketing entropy results for log-concave densities on the real line.  The class $\mathcal{F}$ is not totally bounded with respect to Hellinger distance, but since this metric is invariant to affine transformations, one can consider subclasses of $\mathcal{F}$ with mean and variance restrictions.  More precisely, for $\xi \geq 0$ and $\eta \in (0,1)$, let
\begin{align*}
\tilde{\mathcal{F}}^{\xi, \eta} &:= \{f \in \mathcal{F}:|\mu_f| \leq \xi,|\sigma_f^2-1| \leq \eta\}.
\end{align*}
\citet[][Theorem~4]{KimSamworth2016} proved that   
\begin{equation}
\label{Eq:GlobalBE}
H_{[]}(\epsilon, \tilde{\mathcal{F}}^{1,1/2}, \shel) \leq C\epsilon^{-1/2};
\end{equation}
see also \citet[][Theorem~3.1]{DossWellner} for a closely related result with different but similar restrictions on the class $\mathcal{F}$.  Thus Theorem~\ref{ment} reveals that when $f_0 \in \mathcal{F}$ is close to some $f_1 \in \mathcal{F}^1$ with $f_0 \ll f_1$, and when $\delta > 0$ is small, the local bracketing entropy is much smaller than the global bracketing entropy described by~\eqref{Eq:GlobalBE}.  

The proof of Theorem~\ref{ment} is lengthy, but the main ideas are as follows. By a triangle inequality, one can show that it suffices to prove the result for $f_0 \in \F^1$.  In fact, by an affine transformation, it is enough to consider $f_0$ belonging to one of three canonical forms within the class $\F^1$.  When $f_0 \in \F^1$, we have $\upsilon = 0$, and we can exploit natural boundedness properties enjoyed by $f \in \F(f_0, \delta)$ when $f \ll f_0$ and $\delta>0$ is sufficiently small. For example, when $f_0$ is the uniform density on $[0, 1]$, it is possible to show (see Lemma~\ref{ub} in Section~\ref{Sec:Auxment}) that such $f$ satisfy $\log f(x) \leq C  \delta$ for all $x \in [0, 1]$ and $\log f(x) \geq - C \delta \max\{x^{-1/2}, (1 - x)^{-1/2}\}$ whenever $\min(x, 1 - x) \geq 4 \delta^2$. These boundedness properties allow us to apply bracketing entropy bounds for bounded log-concave functions developed in Propositions~\ref{mb1} and~\ref{meni} in Section~\ref{Sec:Auxment} to deduce the result.

Theorem~\ref{ment} enables us to prove the following risk bound for the log-concave maximum likelihood estimator when the true density is close to $\mathcal{F}^1$, a key step in proving Theorem~\ref{fik}:
\begin{theorem}\label{k1}
  There exists a universal constant $C > 0$ such that for every $n \geq 2$ and $f_0 \in \F$, we have 
  \begin{equation}
    \label{k1.eq}
    \E_{f_0} \kld(\hat{f}_n, f_0) \leq \frac{C}{n} \log^{5/4} n + \inf_{f_1 \in \F^1: f_0 \ll f_1} \hel(f_0, f_1). 
  \end{equation}
\end{theorem}
Since $\hel(f_0, f_1) \leq \kl(f_0, f_1)$ and since $\kl(f_0, f_1) = \infty$ unless $f_0 \ll f_1$, the inequality given in~\eqref{k1.eq} is stronger than the inequality obtained by replacing the infimum on the right-hand side of~\eqref{fik.eq} by $k = 1$.  

\section{Proofs of main results}
\label{Sec:ProofsMain}

\subsection{Proofs from Section~\ref{Sec:TV} and alternative total variation bound}

We first present the proof of Theorem~\ref{ct}, and then give the proof of Lemma~\ref{miu}, on which it relies.
\begin{proof}[Proof of Theorem~\ref{ct}]
Let $\hat{F}_n$ and $F_0$ denote the distribution functions of $\hat{f}_n$ and $f_0$ respectively, and let $\mathbb{F}_n$ denote the empirical distribution function of $X_1, \ldots, X_n$.  Fix $f_{\alpha,s_1,s_2} \in \mathcal{F}^1$ with $(\alpha,s_1,s_2) \in \mathcal{T}$, and let $F_{\alpha,s_1,s_2}$ denote its corresponding distribution function.  Then $\{x:\hat{f}_n(x) \geq f_{\alpha,s_1,s_2}(x)\} = \{x:\log \hat{f}_n(x) \geq \log f_{\alpha,s_1,s_2}(x)\}$ is an interval.  It follows that
\begin{align}
\label{Eq:dTV}
d_{\mathrm{TV}}(\hat{f}_n, f_{\alpha,s_1,s_2}) &= \int_{x:\hat{f}_n(x) \geq f_{\alpha,s_1,s_2}(x)} \bigl\{\hat{f}_n(x) - f_{\alpha,s_1,s_2}(x)\bigr\} \, dx \nonumber \\
&= \sup_{s \leq t} \int_s^t \bigl\{\hat{f}_n(x) - f_{\alpha,s_1,s_2}(x)\bigr\} \, dx \nonumber \\
&= \sup_{s \leq t} \bigl[\hat{F}_n(t) - F_{\alpha,s_1,s_2}(t) - \{\hat{F}_n(s) - F_{\alpha,s_1,s_2}(s)\}\bigr] \leq 2\|\hat{F}_n - F_{\alpha,s_1,s_2}\|_\infty.
\end{align}
Hence, writing $d_{\mathrm{TV}} := d_{\mathrm{TV}}(f_{\alpha,s_1,s_2},f_0)$,
\begin{align*}
d_{\mathrm{TV}}(\hat f_n, f_0) &\leq d_{\mathrm{TV}}(\hat f_n, f_{\alpha,s_1,s_2}) + d_{\mathrm{TV}} \\
&\leq 2\|\hat{F}_n - F_{\alpha,s_1,s_2}\|_\infty + d_{\mathrm{TV}} \\
   &\leq 2\rho(|\kappa|)\|\mathbb{F}_n - F_{\alpha,s_1,s_2}\|_\infty + d_{\mathrm{TV}} \\
   &\leq 2\rho(|\kappa|)\|\mathbb{F}_n-F_0\|_\infty + \{1+2\rho(|\kappa|)\} d_{\mathrm{TV}},
\end{align*}
where $\kappa := \alpha(X_{(n)} - X_{(1)})$.  Here, the penultimate inequality follows from Lemma~\ref{miu}, and the final one follows by the triangle inequality and the fact that $\|F-G\|_\infty \leq d_{\mathrm{TV}}(f,g)$ for any densities $f$ and $g$ with corresponding distribution functions $F$ and $G$ respectively.  It is now convenient to introduce $Y_1,\ldots,Y_n \stackrel{\mathrm{iid}}{\sim} f_{\alpha,s_1,s_2}$, with $Y_{(1)} := \min_i Y_i$ and $Y_{(n)} := \max_i Y_i$.  Then, writing $\kappa^* := \alpha(s_2-s_1)$ and $d_{\mathrm{KS}}^{(n)} := \|F_{\alpha,s_1,s_2}^n - F_0^n\|_\infty + \|(1-F_{\alpha,s_1,s_2})^n - (1-F_0)^n\|_\infty$,
\begin{align*}
\mathbb{P}_{f_0}(|\kappa|> |\kappa^*|) &\leq \mathbb{P}_{f_0}(X_{(n)} > s_2) + 
\mathbb{P}_{f_0}(X_{(1)} < s_1) \\
&\leq \mathbb{P}_{f_{\alpha,s_1,s_2}}(Y_{(n)} > s_2) + \mathbb{P}_{f_{\alpha,s_1,s_2}}(Y_{(1)} < s_1) + d_{\mathrm{KS}}^{(n)} = d_{\mathrm{KS}}^{(n)}.
\end{align*}
Since $\rho$ is strictly increasing, we can therefore apply the Dvoretzky--Kiefer--Wolfowitz inequality \citep{dvoretzky1956asymptotic} with the sharp constant of \citet{massart1990tight} to conclude that for any $t \geq 0$,
\begin{align*}
\mathbb{P}_{f_0}\bigl[d_{\mathrm{TV}}(\hat f_n, f_0) \geq t + \{1+2\rho(|\kappa|)\}d_{\mathrm{TV}}\bigr] &\leq \mathbb{P}_{f_0}\bigl\{2\rho(|\kappa^*|)\|\mathbb{F}_n-F_0\|_\infty \geq t\bigr\} + \mathbb{P}_{f_0}(|\kappa|> |\kappa^*|) \\
&\leq 2\exp\biggl(-\frac{nt^2}{2\rho^2(|\kappa^*|)}\biggr) + d_{\mathrm{KS}}^{(n)}.
\end{align*}
For the other bound~\eqref{ct.eq2}, note first that if $B \geq 2$ and $\alpha < 0$, then $s_1 > -\infty$, so
\begin{align*}
\mathbb{P}_{f_0}\biggl(|\kappa|>\frac{B}{2}\log n\biggr) &\leq \mathbb{P}_{f_0}\biggl(X_{(n)} > s_1 -\frac{B\log n}{2\alpha}\biggr) + \mathbb{P}_{f_0}(X_{(1)} < s_1) \\
&\leq \mathbb{P}_{f_{\alpha,s_1,s_2}}\biggl(Y_{(n)} > s_1 -\frac{B\log n}{2\alpha}\biggr) + \mathbb{P}_{f_{\alpha,s_1,s_2}}(Y_{(1)} < s_1) + d_{\mathrm{KS}}^{(n)} \\
&= 1-\biggl(\frac{1-n^{-B/2}}{1 - e^{\alpha (s_2-s_1)}}\biggr)^n + d_{\mathrm{KS}}^{(n)} \\
&\leq 1-(1-n^{-B/2})^n + d_{\mathrm{KS}}^{(n)} \leq n^{-(B/2-1)} + d_{\mathrm{KS}}^{(n)},
\end{align*}
where the final inequality follows because $1 - x \leq (1-x/n)^n$ for $x \in [0,1)$ (this can be proved by taking logarithms and examining the Taylor series).  A very similar calculation yields the same bound when $\alpha > 0$.  Recalling that $\rho(x) \leq \max(3,2x)$, it follows that if $t \geq 0$ and $(\alpha,s_1,s_2) \in \mathcal{T}$ with $\alpha \neq 0$, then provided $B \geq 2$ and $B \log n \geq 3$,
\begin{align*}
\mathbb{P}_{f_0}\bigl\{d_{\mathrm{TV}}(\hat f_n, f_0) \geq t + (1 + 2B&\log n)d_{\mathrm{TV}}\bigr\} \\
&\leq \mathbb{P}_{f_0}\bigl\{\rho(|\kappa|) > B \log n\bigr\} + \mathbb{P}_{f_0}\bigl\{2B \log n\|\mathbb{F}_n-F_0\|_\infty \geq t\bigr\} \\
&\leq n^{-(B/2-1)} + d_{\mathrm{KS}}^{(n)} + 2\exp\biggl(-\frac{nt^2}{2B^2 \log^2 n}\biggr),
\end{align*}
where the final inequality follows by another application of the Dvoretzky--Kiefer--Wolfowitz inequality.  Taking $B = 3$ and $n \geq 3$ therefore yields~\eqref{ct.eq2}.

Writing $s^* := (2\log 2)^{1/2}\rho(|\kappa^*|)/n^{1/2}$, it follows that
\begin{align*}
\E_{f_0}d_{\mathrm{TV}}(\hat{f}_n,f_0) &= \mathbb{E}\bigl\{d_{\mathrm{TV}}(\hat{f}_n,f_0)\mathbbm{1}_{\{|\kappa| \leq |\kappa^*|\}}\bigr\} + \mathbb{E}\bigl\{d_{\mathrm{TV}}(\hat{f}_n,f_0)\mathbbm{1}_{\{|\kappa| > |\kappa^*|\}}\bigr\} \\
&\leq \E_{f_0}\bigl(\bigl[d_{\mathrm{TV}}(\hat{f}_n,f_0) - \{1+2\rho(|\kappa|)\}d_{\mathrm{TV}}\bigr]\mathbbm{1}_{\{|\kappa| \leq |\kappa^*|\}}\bigr) + \{1+2\rho(|\kappa^*|)\}d_{\mathrm{TV}} + d_{\mathrm{KS}}^{(n)} \\
&\leq s^* + 2\int_{s^*}^\infty \exp\biggl(-\frac{ns^2}{2\rho^2(|\kappa^*|)}\biggr) \, ds + \{1+2\rho(|\kappa^*|)\}d_{\mathrm{TV}} + d_{\mathrm{KS}}^{(n)} \\
&\leq \frac{2\rho(|\kappa^*|)}{n^{1/2}} + \{1+2\rho(|\kappa^*|)\}d_{\mathrm{TV}} + d_{\mathrm{KS}}^{(n)}.
\end{align*}
On the other hand, writing $s' := 3(2\log 2)^{1/2}n^{-1/2}\log n$, we also have
\begin{align*}
\E_{f_0}d_{\mathrm{TV}}(\hat{f}_n,f_0) &\leq \frac{1}{n^{1/2}} + s' + 2\int_{s'}^\infty \exp\Bigl(-\frac{ns^2}{18\log^2 n}\Bigr) \, ds + (1+6\log n)d_{\mathrm{TV}} + d_{\mathrm{KS}}^{(n)} \\
&\leq \frac{6\log n}{n^{1/2}} + (1+6\log n)d_{\mathrm{TV}} + d_{\mathrm{KS}}^{(n)},
\end{align*}
for $n \geq 5$.  Since these inequalities hold for any $f_{\alpha,s_1,s_2} \in \mathcal{F}_*$, the conclusion follows.
\end{proof}
\begin{proof}[Proof of Lemma~\ref{miu}]
This proof has some similarities with that of \citet[][Lemma~1]{dumbgen2007marshall}.  We define the set of \emph{knots} of $\hat{f}_n$ by  
\[
  \scs := \bigl\{t \in (X_{(1)}, X_{(n)}) : \hat{f}_n'(t-) \neq
    \hat{f}_n'(t+) \bigr\} \cup \{X_{(1)}, X_{(n)}\}
\]
where $X_{(1)}$ and $X_{(n)}$ denote the smallest and largest order statistics of the data $X_1, \dots, X_n$. By e.g.~\citet[][Theorem~2.1]{DumbgenRufibach2009}, $\scs \subseteq \{X_1, \dots, X_n\}$, and we therefore write $\scs = \{t_0, \ldots, t_k \}$ for some $k \in \{1,\ldots,n-1\}$ where $X_{(1)} = t_0 < \ldots < t_k = X_{(n)}$.  We first write the left-hand side of \eqref{miu.eq2} as 
\begin{equation*}
  \max \biggl\{\sup_{x < t_0} |\hat{F}_n(x) - F_0(x)|, \max_{i \in \{0,1,\ldots,k-1\}} \sup_{x \in [t_i,t_{i+1})} |\hat{F}_n(x) - F_0(x)|,
    \sup_{x \geq t_k} |\hat{F}_n(x) - F_0(x)| \biggr\}.
\end{equation*}
Observe now that $\hat{F}_n(x) = 0 = \mathbb{F}_n(x)$ for $x < t_0$ and  $\hat{F}_n(x) = 1 = \mathbb{F}_n(x)$ for $x \geq t_k$. It therefore follows that in order to establish \eqref{miu.eq2}, we need only establish the two statements
\begin{align}\label{tp1}
  \sup_{x \in [t_i,t_{i+1})} \{\hat{F}_n(x) - F_0(x)\} &\leq \rho(\kappa_+)\|\mathbb{F}_n - F_0\|_\infty \quad \text{for every $i=0,1,\ldots,k-1$}, \\
    \label{tp2}
    \inf_{x \in [t_i,t_{i+1})} \{\hat{F}_n(x) - F_0(x)\} &\geq -\rho(\kappa_-)\|\mathbb{F}_n - F_0\|_\infty \quad \text{for every $i=0,1,\ldots,k-1$},
\end{align}
where $\kappa_+ := \max(\kappa,0)$ and $\kappa_{-} := \max(-\kappa,0)$.  It is convenient to prove the second statement first.  We may assume that the infimum of $\hat{F}_n(x) - F_0(x)$ over $x \in [X_{(1)},X_{(n)}]$ is attained at $r \in [t_i,t_{i+1})$, say, for some $i \in \{0,1,\ldots,k-1\}$ and let $a := t_i$ and $b := t_{i+1}$.  By hypothesis, there exist $\alpha_0 \in \mathbb{R}$ and a concave function $h_0:[a,b] \rightarrow [0,\infty)$ such that $f_0(x) = e^{\alpha_0 x}h_0(x)$ for $x \in [a,b]$.  Moreover, there exist $\alpha,\beta \in \mathbb{R}$ such that $\hat{f}_n(x) = \exp(\alpha x+\beta)$ for $x \in [a,b]$.  It follows that if we define 
\[
g(x) := e^{(\alpha - \alpha_0)x + \beta} - h_0(x) = e^{-\alpha_0 x}\{\hat{f}_n(x) - f_0(x)\},
\]
then $g$ is convex on $[a,b]$ and $g(r) = 0$.  Moreover, defining $G(x) := \hat{F}_n(x) - F_0(x)$, we have
\begin{equation*}
  G(x) = c + \int_a^x e^{\alpha_0 t} g(t) \, dt \qt{for $x \in [a,b]$,}
\end{equation*}
where $c := \int_{-\infty}^a \hat{f}_n(t) - f_0(t) \, dt$.  We may therefore apply either inequality~\eqref{coni4} or inequality~\eqref{coni2} in Lemma~\ref{coni} in the Appendix (depending on whether or not $\alpha_0 = 0$) to obtain that for every $x \in [r,b]$,
\begin{equation*}
  G(x) \leq \left\{ \begin{array}{ll} G(r) + \frac{(x - r)^2}{(b - r)^2} \{G(b) - G(r)\} & \mbox{if $\alpha_0 = 0$} \\
 G(r) + \frac{1+e^{\alpha_0(x-r)}\{\alpha_0(x-r)-1\}}{1+e^{\alpha_0(b-r)}\{\alpha_0(b-r)-1\}}\{G(b)-G(r)\} & \mbox{if $\alpha_0 \neq 0$.} \end{array} \right.
\end{equation*}
Integrating from $x = r$ to $x = b$, writing $A := \alpha_0(b-r)$ and recalling the definition of the function $q$ in~\eqref{qfunc}, we deduce that 
\begin{equation}\label{oop}
  G(r) \geq \frac{1}{b-r}\frac{1}{1-q(-A)} \int_r^b G(x) \, dx -  \frac{q(-A)}{1-q(-A)} G(b).
\end{equation}
Now \citet[Theorem 2.4]{DumbgenRufibach2009} yields that
\begin{equation}\label{dr}
  \int_{-\infty}^t \hat{F}_n(x) \, dx \leq \int_{-\infty}^t \mathbb{F}_n(x) \, dx ~~
  \text{ and } ~~ \int_{-\infty}^s \hat{F}_n(x) \, dx =
  \int_{-\infty}^s \mathbb{F}_n(x) \, dx 
\end{equation}
for every $t \in \R$ and $s \in \scs$.  Moreover, \citet[Remark 2.8]{DSS2011} gives that
\begin{equation}\label{srk}
  \mathbb{F}_n(x) - \frac{1}{n} \leq \hat{F}_n(x) \leq \mathbb{F}_n(x) \qt{for every $x \in \scs$}. 
\end{equation} 
It follows from~\eqref{oop},~\eqref{dr} and~\eqref{srk} that
\begin{align*}
 G(r)  &\geq \frac{1}{b-r}\frac{1}{1-q(-A)} \biggl\{\int_{-\infty}^b \hat{F}_n(x) \, dx - \int_{-\infty}^r \hat{F}_n(x) \, dx - \int_r^b F_0(x) \, dx\biggr\} -  \frac{q(-A)}{1-q(-A)} G(b) \\
&\geq \frac{1}{b-r}\frac{1}{1-q(-A)} \int_r^b \{\mathbb{F}_n(x) - F_0(x)\} \, dx - \frac{q(-A)}{1-q(-A)}\{\mathbb{F}_n(b) - F_0(b)\} \\
 &\geq -\rho(-A) \|\mathbb{F}_n- F_0\|_\infty \geq -\rho(\kappa_{-})\|\mathbb{F}_n- F_0\|_\infty.
\end{align*}
This establishes~\eqref{tp2}.  For~\eqref{tp1}, let $Y_i := -X_i$, let $\hat{h}_n$ denote the log-concave maximum likelihood estimator based on $Y_1,\ldots,Y_n$, and let $\hat{H}_n$ denote its corresponding distribution function, so that by affine equivariance of the log-concave maximum likelihood estimator \citep[Remark~2.4]{DSS2011}, we have $\hat{h}_n(y) = \hat{f}_n(-y)$ and $\hat{H}_n(y) = 1 - \hat{F}_n(-y)$.  Similarly, let $h_0(y) := f_0(-y)$ (so $h_0$ is concave on the convex hull of $Y_1,\ldots,Y_n$), and let $H_0$ denote the distribution function corresponding to the density $h_0$, so that $H_0(y) = 1 - F_0(-y)$.  Finally, let $\mathbb{H}_n$ denote the empirical distribution function corresponding to $Y_1,\ldots,Y_n$, so $\mathbb{H}_n(y) = n^{-1}\sum_{i=1}^n \mathbbm{1}_{\{Y_i \leq y\}} = 1 - \lim_{z \searrow y} \mathbb{F}_n(-z)$.  Then for any two consecutive knots $a$ and $b$ of $\hat{f}_n$,
\begin{align*}
\sup_{x \in (a,b]} \{\hat{F}_n(x) - F_0(x)\} &= \sup_{x \in (a,b]} -\{\hat{H}_n(-x) - H_0(-x)\} = -\inf_{y \in [-b,-a)} \{\hat{H}_n(y) - H_0(y)\} \\
&\leq \rho(\kappa_{+})\|\mathbb{H}_n - H_0\|_\infty = \rho(\kappa_{+})\|\mathbb{F}_n - F_0\|_\infty,
\end{align*}
as required, where the inequality follows from an application of~\eqref{tp2} to the transformed data $Y_1,\ldots,Y_n$, noting that $-\alpha_0(Y_{(n)} - Y_{(1)}) = -\alpha_0(X_{(n)} - X_{(1)}) = -\kappa$.  
\end{proof}
Recall the definition of $\mathcal{F}_*$ in~\eqref{Eq:F_*}.  We now provide a result which improves the bounds given in Theorem~\ref{ct} in the special case where the true density belongs to the class $\mathcal{F}_*$.
\begin{proposition}
\label{Prop:ct}
Let $n \geq 5$, let $X_1,\ldots,X_n \stackrel{\mathrm{iid}}{\sim} f_0 \in \mathcal{F}_*$, and let $\hat{f}_n$ denote the corresponding log-concave maximum likelihood estimator.  Fix an arbitrary $f_{\alpha,s_1,s_2} \in \mathcal{F}^1$, write $\kappa^* := \alpha(s_2-s_1)$, let $d_{\mathrm{TV}} := d_{\mathrm{TV}}(f_{\alpha,s_1,s_2},f_0)$ and let $d_{\mathrm{KS}}^{(n)} := \|F_{\alpha,s_1,s_2}^n - F_0^n\|_\infty + \|(1-F_{\alpha,s_1,s_2})^n - (1-F_0)^n\|_\infty$, where $F_{\alpha,s_1,s_2}$ and $F_0$ are the distribution functions corresponding to $f_{\alpha,s_1,s_2}$ and $f_0$ respectively.  Then, for $t \geq 0$,
  \begin{equation*}
    \P_{f_0}\bigl\{d_{\mathrm{TV}}(\hat{f}_n, f_0) \geq t + 3d_{\mathrm{TV}}\bigr\} \leq \min\biggl\{2 e^{-\frac{nt^2}{2\rho^2(|\kappa^*|)}} \, , \, \frac{1}{n^{1/2}} + 2e^{-\frac{nt^2}{18\log^2 n}}\biggr\} + d_{\mathrm{KS}}^{(n)}.
  \end{equation*}
Moreover,
\[
\E_{f_0} d_{\mathrm{TV}}(\hat{f}_n, f_0) \leq \inf_{f_{\alpha,s_1,s_2} \in \mathcal{F}^1}\biggl\{\frac{\min\{2\rho(|\kappa^*|),6\log n\}}{n^{1/2}} + 3d_{\mathrm{TV}} + d_{\mathrm{KS}}^{(n)}\biggr\}.
\]
\end{proposition}
\begin{proof}
For any $f_{\alpha,s_1,s_2} \in \mathcal{F}^1$ with $(\alpha,s_1,s_2) \in \mathcal{T}$ and corresponding distribution function $F_{\alpha,s_1,s_2}$, we have by~\eqref{Eq:dTV} that
\begin{align*}
d_{\mathrm{TV}}(\hat f_n, f_0) &\leq d_{\mathrm{TV}}(\hat f_n, f_{\alpha,s_1,s_2}) + d_{\mathrm{TV}} \\
&\leq 2\|\hat{F}_n - F_{\alpha,s_1,s_2}\|_\infty + d_{\mathrm{TV}} \\
   &\leq 2\|\hat{F}_n - F_0\|_\infty + 2\|F_0 - F_{\alpha,s_1,s_2}\|_\infty + d_{\mathrm{TV}} \\
   &\leq 2\rho(|\kappa|)\|\mathbb{F}_n-F_0\|_\infty + 3d_{\mathrm{TV}},
\end{align*}
where $\kappa := \alpha(X_{(n)}-X_{(1)})$, and where the last line follows again from Lemma~\ref{miu}.  The proof now follows that of Theorem~\ref{ct}, mutatis mutandis, so we omit the details for brevity.
\end{proof}

\section{Proofs from Section~\ref{Sec:kaffine}}
\label{Sec:kaffineproofs}

\begin{proof}[Proof of Theorem~\ref{fik}]
Fix a density $f_0 \in \F$. Also fix $k \in \mathbb{N}$ and an arbitrary density $f \in \F^k$ such that $\kl(f_0, f) < \infty$. Note that this implies that $f_0 \ll f$. Suppose that $I_1, \dots, I_k$ is a partition of the support of $f$ into maximal intervals such that $\log f$ is affine on each $I_j$. Since $f_0$ is absolutely continuous with respect to $f$, it follows that $\sum_{j=1}^k p_j = 1$, where $p_j := \int_{I_j} f_0$.  For $j=1,\ldots,k$, we also let $N_j := \sum_{i=1}^n  \mathbbm{1}_{\{X_i \in I_j\}}$, $J_1  := \{j: N_j  \geq 2 \}$ and $J_2 := \{j: N_j \leq 1\}$. Observe that the sets $J_1$ and $J_2$ as well as the integers $N_1,\ldots,N_k$ are random.  We write  
  \begin{align}
\label{Eq:FirstBound}
    \kld(\hat{f}_n, f_0) &= \frac{1}{n} \sum_{i=1}^n \log \frac{\hat{f}_n(X_i)}{f_0(X_i)} \nonumber \\
&= \frac{1}{n} \sum_{j \in J_1} \sum_{i: X_i \in I_j} \log \frac{\hat{f}_n(X_i)}{f_0(X_i)} + \frac{1}{n} \sum_{j \in J_2}
    \sum_{i: X_i \in I_j} \log \frac{\hat{f}_n(X_i)}{f_0(X_i)} \nonumber \\
&\leq \frac{1}{n} \sum_{j \in J_1} \sum_{i: X_i \in I_j} \log \frac{\hat{f}_n(X_i)}{f_0(X_i)} + \frac{k}{n} \max_{1 \leq i \leq
      n} \log \frac{\hat{f}_n(X_i)}{f_0(X_i)},
  \end{align}
  where the final inequality follows because $|J_2| \leq k$ and $|N_j| \leq 1$ whenever $j \in J_2$.  To handle the first term, let $\tilde{f}_n$ denote the maximum likelihood estimator based on the data $\{X_i : X_i \in \cup_{j \in J_1} I_j\}$ over the class of all densities $f$ for which $\log f$ is concave on each of the intervals $\{I_j:j \in J_1\}$.  Since $\log \hat{f}_n$ is concave on each $I_j$ and since $\int_{-\infty}^\infty \hat{f}_n(x) \mathbbm{1}_{\{x \in \cup_{j \in J_1} I_j\}} \, dx \leq 1$,  it follows that 
  \begin{equation}
\label{Eq:tilde}
    \sum_{j \in J_1} \sum_{i: X_i \in I_j} \log
    \hat{f}_n(X_i) \leq \sum_{j \in J_1} \sum_{i: X_i \in I_j} \log
    \tilde{f}_n(X_i).
  \end{equation}
Writing $M_1 := \sum_{j \in J_1} N_j$, we claim that
\begin{equation}\label{cm}
\tilde{f}_n(x) = \frac{N_j}{M_1}\hat{f}^{(j)}(x) \qt{for $x \in I_j$ with $j \in J_1$,} 
\end{equation}
where $\hat{f}^{(j)}$ denotes the log-concave maximum likelihood estimator based on $\{X_i:X_i \in I_j\}$.  To see this, let $\bar{\Phi}$ denote the class of functions $\phi:\mathbb{R} \rightarrow [-\infty,\infty)$ that are concave on each $I_j$ for $j \in J_1$ and that satisfy $\phi(x) \rightarrow -\infty$ as $|x| \rightarrow \infty$. Now, $\log \tilde{f}_n$ maximises
\[
L(\phi) := \frac{1}{M_1}\sum_{j\in J_1}\sum_{i:X_i \in I_j} \phi(X_i) - \sum_{j \in J_1} \int_{I_j} e^\phi = \sum_{j \in J_1} \frac{N_j}{M_1} \biggl\{\frac{1}{N_j}\sum_{i:X_i \in I_j} \phi(X_i) - \int_{I_j} e^{\phi + \log(M_1/N_j)}\biggr\} 
\]
over $\phi \in \bar{\Phi}$.  For $j \in J_1$, let $\bar{\Phi}_j$ denote the set of functions $\phi:I_j \rightarrow [-\infty,\infty)$ that are restrictions of functions in $\bar{\Phi}$ to $I_j$.  Then, on each interval $I_j$ with $j \in J_1$, we have
\begin{align*}
\log \tilde{f}_n &= \argmax_{\phi \in \bar{\Phi}_j} \biggl\{\frac{1}{N_j}\sum_{i:X_i \in I_j} \phi(X_i) - \int_{I_j} e^{\phi + \log(M_1/N_j)}\biggr\} \\ 
&= \argmax_{\tilde{\phi} \in \bar{\Phi}_j} \biggl\{\frac{1}{N_j}\sum_{i:X_i \in I_j} \tilde{\phi}(X_i) - \int_{I_j} e^{\tilde{\phi}}\biggr\} - \log \frac{M_1}{N_j} = \log \hat{f}^{(j)} - \log \frac{M_1}{N_j}, 
\end{align*}
which establishes the claim \eqref{cm}.  Let $f_0^{(j)}(x) := \frac{1}{p_j}f_0(x)\mathbbm{1}_{\{x \in I_j\}}$. We deduce from~\eqref{Eq:tilde} and~\eqref{cm} that
\begin{align}
\label{Eq:ThreeTerms}
\frac{1}{n} &\mathbb{E}_{f_0} \biggl\{\sum_{j \in J_1} \sum_{i: X_i \in I_j} \log \frac{\hat{f}_n(X_i)}{f_0(X_i)}\biggr\} \leq \frac{1}{n} \mathbb{E}_{f_0} \biggl\{\sum_{j \in J_1} \sum_{i: X_i \in I_j} \log \frac{N_j\hat{f}^{(j)}(X_i)/M_1}{p_jf_0^{(j)}(X_i)}\biggr\}
    \nonumber \\ 
&= \frac{1}{n} \mathbb{E}_{f_0} \biggl\{\sum_{j \in J_1} \sum_{i: X_i \in I_j} \log \frac{\hat{f}^{(j)}(X_i)}{f_0^{(j)}(X_i)}\biggr\} +
\mathbb{E}_{f_0}\biggl(\sum_{j \in J_1} \frac{N_j}{n} \log \frac{N_j}{np_j}\biggr) + \mathbb{E}_{f_0}\biggl(\frac{M_1}{n}\log \frac{n}{M_1}\biggr). 
\end{align}
Now let $f^{(j)}(x) := \frac{1}{q_j}f(x)\mathbbm{1}_{\{x \in I_j\}}$, where $q_j := \int_{I_j} f$, and note both that $f^{(j)} \in \mathcal{F}^1$ and $f_0^{(j)} \ll f^{(j)}$.  To evaluate the first expectation on the right-hand side of~\eqref{Eq:ThreeTerms}, we condition on the set of random variables $\{N_j: j=1,\ldots,k\}$. After this conditioning, and since $N_j \geq 2$ for every $j \in J_1$, we can apply the risk bound in Theorem \ref{k1} for each $f_0^{(j)}$ to deduce that 
\begin{align}
  \label{ft}
   \frac{1}{n} \mathbb{E}_{f_0} \biggl\{\sum_{j \in J_1} \sum_{i: X_i \in I_j} \log \frac{\hat{f}^{(j)}(X_i)}{f_0^{(j)}(X_i)}\biggr\} &\leq \frac{1}{n} \E_{f_0} \sum_{j \in J_1} N_j \biggl\{\frac{C}{N_j} \log^{5/4} N_j + \inf_{f_1 \in \F^1:f_0^{(j)} \ll f_1} \hel(f_0^{(j)}, f_1) \biggr\} \nonumber \\
&\leq \frac{Ck}{n} \log^{5/4} n + \sum_{j = 1}^k p_j \hel(f_0^{(j)}, f^{(j)}) \nonumber \\
&\leq \frac{Ck}{n} \log^{5/4} n + \kl(f_0, f),
\end{align}
where the final inequality follows because
\[
\sum_{j = 1}^k p_j \hel(f_0^{(j)}, f^{(j)}) \leq \sum_{j=1}^k p_j \kl(f_0^{(j)}, f^{(j)}) \leq \sum_{j=1}^k p_j \kl(f_0^{(j)}, f^{(j)}) + \sum_{j=1}^k p_j \log \frac{p_j}{q_j} = \kl(f_0, f). 
\]
To handle the second term on the right-hand side of~\eqref{Eq:ThreeTerms}, we use the facts that $\log x \leq x-1$ for $x > 0$, $N_j \sim \mathrm{Bin}(n,p_j)$ and $N_j \log N_j = 0$ for $j \in J_2$, to obtain
\begin{align}
\label{Eq:SecondTerm}
\mathbb{E}_{f_0}\biggl(\sum_{j \in J_1} \frac{N_j}{n} \log \frac{N_j}{np_j}\biggr) \leq \sum_{j=1}^k \mathbb{E}_{f_0} \biggl\{\frac{N_j}{n}\biggl(\frac{N_j}{np_j} - 1\biggr)\biggr\} - \mathbb{E}_{f_0}\biggl(\sum_{j \in J_2} \frac{N_j}{n} \log \frac{N_j}{np_j}\biggr)
\leq \frac{k}{n} + \frac{k}{n}\log n.
\end{align}
Finally, for the third term on the right-hand side of~\eqref{Eq:ThreeTerms}, we temporarily assume that $k < n/2$ and note that in that case, $M_1 = n - \sum_{j \in J_2} N_j \geq n - k > n/2$, so in particular $\mathbb{E}(M_1/n) \geq 1 - k/n$.  Applying Jensen's inequality, the fact that $x \mapsto -x \log x$ is decreasing on $[1/2,1]$ and the fact that $-\log(1 - x) \leq x + x^2$ for $x \in (0,1/2]$, we deduce that  
\begin{equation}
\label{Eq:ThirdTerm}
\mathbb{E}_{f_0}\biggl(\frac{M_1}{n}\log\frac{n}{M_1}\biggr)
\leq - \biggl(1 - \frac{k}{n} \biggr) \log \biggl( 1 - \frac{k}{n}
\biggr) \leq - \log \biggl(1 - \frac{k}{n} \biggr) \leq \frac{k}{n} + 
\frac{k^2}{n^2} \leq \frac{2k}{n}. 
\end{equation}
Thus, when $k < n/2$, the conclusion of the theorem follows from~\eqref{Eq:FirstBound},~\eqref{Eq:ThreeTerms},~\eqref{ft},~\eqref{Eq:SecondTerm} and~\eqref{Eq:ThirdTerm}, together with Lemma~\ref{ist}, which controls the expected value of the second term in~\eqref{Eq:FirstBound}.  When $k \geq n/2$, we can apply Lemma~\ref{ist} again to conclude that 
\[
\E \kld(\hat{f}_n, f_0) = \frac{1}{n} \sum_{i=1}^n \E_{f_0}
\biggl\{\log \frac{\hat{f}_n(X_i)}{f_0(X_i)}\biggr\}
\leq \E_{f_0} \biggl\{\max_{1 \leq i \leq n} \log
\frac{\hat{f}_n(X_i)}{f_0(X_i)}\biggr\} \leq C \log n,  
\]
as required. This completes the proof of Theorem \ref{fik}. 
\end{proof}

\begin{proof}[Proof of Theorem~\ref{ment}]
We consider first the case where $\upsilon = 0$, so that $f_0 \in \mathcal{F}^1$.  In that case, recalling the parametrisation $\mathcal{F}^1 = \{f_{\alpha,s_1,s_2}:(\alpha,s_1,s_2) \in \mathcal{T}\}$ used in Section~\ref{Sec:TV}, the affine invariance of the Hellinger distance, together with Lemma~\ref{affi} in Section~\ref{Sec:Auxment}, shows that we may assume without loss of generality that $f_0$ is of one of the following three forms:
\begin{enumerate}
\item $f_0 = f_{0,0,1}$;
\item $f_0 = f_{-\alpha,0,1}$ for some $\alpha \in (0,18)$;
\item $f_0 = f_{-1,0,a}$ for some $a \in [18,\infty]$.  
\end{enumerate}   
We refer to these three forms as `uniform', `exponential conditioned on $[0,1]$' and `truncated exponential' respectively, and treat the three cases separately.

\emph{The case where $f_0$ is uniform}: Fix $\delta \in (0,2^{-5/2}]$.  Observe first that for every $\epsilon > 0$, we have
\begin{align}\label{kal}
  H_{[]}\bigl(2^{1/2} &\epsilon, \F(f_0, \delta), \shel, [0, 1]\bigr) \leq H_{[]}\bigl(\epsilon, \F(f_0, \delta), \shel, [0, 1/2]\bigr) +
  H_{[]}\bigl(\epsilon, \F(f_0, \delta), \shel, [1/2, 1]\bigr) \nonumber \\
  &= 2H_{[]}\bigl(\epsilon, \F(f_0, \delta), \shel, [0,1/2]\bigr) \nonumber \\
  &\leq 2H_{[]}\bigl(\epsilon/2^{1/2}, \F(f_0, \delta), \shel,  [0, 4\delta^2]\bigr) + 2H_{[]}\bigl(\epsilon/2^{1/2}, \F(f_0, \delta), \shel,  [4\delta^2, 1/2]\bigr).
\end{align}
We bound the two terms on the right-hand side of~\eqref{kal} separately. For the first term, we use inequality \eqref{ub1.eq} in Lemma~\ref{ub} in Section~\ref{Sec:Auxment}, which gives that $\sup_{x \in [0,1]} \log f(x) \leq 2^{13/2} \delta \leq 16$ for every $f \in \F(f_0, \delta)$. From this, Proposition~\ref{meni} in Section~\ref{Sec:Auxment} and the fact that $\delta \in (0,2^{-5/2}]$, we therefore obtain   
\begin{equation}\label{te.1}
  H_{[]}(\epsilon/2^{1/2}, \F(f_0, \delta), \shel, [0, 4\delta^2]) \leq C (1 + 2^{13/4} \delta^{1/2}) \frac{e^4(4\delta^2)^{1/4}}{\epsilon^{1/2}} \leq C\frac{\delta^{1/2}}{\epsilon^{1/2}},
\end{equation}
which takes care of the first term in~\eqref{kal}.  For the second, term, let $\eta_j := 4 \delta^2 2^j$ for $j = 0, 1, \dots, l$  where $l$ is the largest integer for which $4 \delta^2 2^l < 1/2$. Also let $\eta_{l+1} := 1/2$. By Lemma~\ref{ub}, for every $f \in \F(f_0, \delta)$ and $j = 0,1,\ldots,l$, we have that 
\begin{equation*}
-\frac{4 \delta}{\eta_j^{1/2}}  \leq \log f(x) \leq 2^{13/2} \delta \leq 16 \qt{for every $x \in [\eta_j, \eta_{j+1}]$}. 
\end{equation*}
Set $\epsilon_j := \epsilon/(2l+2)^{1/2}$.  Then by Proposition~\ref{mb1} in Section~\ref{Sec:Auxment},
\begin{align*}
 H_{[]}\bigl(\epsilon/2^{1/2}, \F(f_0, \delta), \shel, [4\delta^2, 1/2]\bigr) &\leq \sum_{j=0}^l H_{[]}\bigl(\epsilon_j, \F(f_0, \delta), \shel, [\eta_j, \eta_{j+1}]\bigr) \\
&\leq C \sum_{j=0}^l \biggl(2^{13/2}\delta + \frac{4 \delta}{\eta_j^{1/2}}\biggr)^{1/2}\frac{e^4(\eta_{j+1} - \eta_j)^{1/4}}{\epsilon_j^{1/2}} \\
      &\leq C \delta^{1/2} \sum_{j=0}^l \frac{1}{\epsilon_j^{1/2}} \biggl(\frac{\eta_{j+1} - \eta_j}{\eta_j} \biggr)^{1/4},
\end{align*}
where we have used the fact that $\eta_j \leq 1$ in the final inequality.  Observe now that by our choice of $\eta_j = 4 \delta^2 2^j$ for $j=0,1,\ldots,l$ and $\eta_{l+1} = 1/2 \leq 4 \delta^2 2^{l+1}$, it follows that $\eta_{j+1} - \eta_j \leq \eta_j$ for every $j=0,1,\ldots,l$. We therefore obtain
\begin{align*}
  H_{[]}\bigl(\epsilon/2^{1/2}, \F(f_0, \delta), \shel, [4 \delta^2, 1/2]\bigr)
  \leq C \delta^{1/2}\sum_{j=0}^l \frac{1}{\epsilon_j^{1/2}} \leq C\frac{\delta^{1/2}}{\epsilon^{1/2}}  (l+1)^{5/4} \leq C\frac{\delta^{1/2}}{\epsilon^{1/2}}\log^{5/4}\Bigl(\frac{1}{\delta}\Bigr),
\end{align*}
as required, where the final inequality follows because $4 \delta^2 2^l < 1/2$, so 
\begin{equation*}
  l+1 < \frac{-\log (4 \delta^2)}{\log 2} \leq C \log\Bigl(\frac{1}{\delta}\Bigr).  
\end{equation*}

\emph{The exponential conditioned on $[0, 1]$ case}: 
Now suppose $f_0 = f_{-\alpha,0,1}$ for some $\alpha \in (0,18)$, let $C_\alpha := \alpha (1 - e^{-\alpha})^{-1}$ and again fix $\delta \in (0,2^{-5/2}]$.  For every $f = e^{\phi} \in \F(f_0, \delta)$, we have 
\begin{align*}
  \delta^2 \geq \int_0^1 (e^{\phi(x)/2} - C_\alpha^{1/2}e^{-\alpha x/2})^2 \, dx &= C_\alpha \int_0^1  e^{-\alpha x}
  \biggl(\frac{1}{C_\alpha^{1/2}}e^{\{\phi(x) + \alpha x\}/2} - 1 \biggr)^2 \, dx \\
&\geq C_\alpha e^{-\alpha}  \int_0^1 \biggl(\frac{1}{C_\alpha^{1/2}}e^{\{\phi(x) + \alpha x\}/2} - 1 \biggr)^2 \, dx.
\end{align*}
Write $\tilde{\delta} := \delta e^{\alpha/2}/C_\alpha^{1/2}$, so that
\[
\delta \leq \tilde{\delta} \leq \biggl(\frac{e^{18} - 1}{18}\biggr)^{1/2} \delta.
\]
Thus, by the same argument as for the uniform case, provided $\delta \in \bigl(0,\bigl(\frac{18}{e^{18}-1}\bigr)^{1/2}2^{-5/2}\bigr]$ and given any $\epsilon > 0$, we can find an $\epsilon/C_\alpha^{1/2}$-Hellinger bracketing set $\{[g_{L, j}, g_{U, j}], j = 1, \dots, N\}$ for the class $\{x \mapsto C_\alpha^{-1}f(x) e^{\alpha x} : f \in \F(f_0, \delta)\}$ with
\[
\log N \leq C\log^{5/4} \Bigl(\frac{1}{\tilde{\delta}}\Bigr)\frac{\tilde{\delta}^{1/2}C_\alpha^{1/4}}{\epsilon^{1/2}} \leq C \log^{5/4} \Bigl(\frac{1}{\delta}\Bigr)\frac{\delta^{1/2}}{\epsilon^{1/2}}.
\]
Now let $f_{L,j}(x) := C_\alpha g_{L, j}(x) e^{-\alpha x}$ and $f_{U, j}(x) := C_\alpha g_{U, j}(x)e^{-\alpha x}$ for $j=1,\ldots,N$.  Then
\begin{align*}
  \int_0^1 (f_{U, j}^{1/2} - f_{L, j}^{1/2})^2 &= C_\alpha \int_0^1 e^{-\alpha x} \bigl\{g_{U, j}^{1/2}(x) - g_{L, j}^{1/2}(x)\bigr\}^2 \, dx \leq C_\alpha \int_0^1 (g_{U, j}^{1/2} - g_{L, j}^{1/2})^2 \leq \epsilon^2,   
\end{align*}
so $\{[f_{L, j}, f_{U, j}], j = 1, \dots, N\}$ form an $\epsilon$-Hellinger bracketing set for $\F(f_0, \delta)$, as required.  

\emph{The case where $f_0$ is truncated exponential}:

Now suppose that $f_0 = f_{-1,0,a}$ for some $a \in [18,\infty]$.  Given a function $\phi: \R \rightarrow [-\infty, \infty)$, we define $\tilde{\phi}_a: \R \rightarrow [-\infty, \infty)$ by
\begin{equation}
\label{Eq:phitildea}
  \tilde{\phi}_a(x) := \phi(x) + x + \log(1 - e^{-a}). 
\end{equation}
Let $x_0$ be defined as in the statement of Lemma~\ref{blan} in Section~\ref{Sec:Auxment} and assume that $\delta \leq \kappa := e^{-9}/8$, so that $x_0 \geq 17$. Also let $l = \lfloor x_0 \rfloor$ and $J := \sup\{j \in \mathbb{N} : x_0 + j - l - 1 \leq a\}$.  We define subintervals of $[0,a]$ (or $[0,a)$ when $a=\infty$) by  
\begin{align*}
  S_1 &:= [0,1] \\
  S_j &:= [j-1, \min(j, x_0)] \qt{for $j = 2, \dots, l + 1$} \\
  S_j &:= [x_0 + j - l - 2, \min(x_0 + j - l - 1, a)] \qt{for $j = l+2, \dots, J+1$.} 
\end{align*}
Also let 
\begin{equation}
\label{mise}
  \epsilon_j^2 := \left\{ \begin{array}{ll}\frac{2(1 - e^{-18})}{3} \epsilon^2 & \mbox{for $j=1$,} \\
\frac{2(1 - e^{-18})}{3} \frac{e^{j-1}\epsilon^2}{l} & \mbox{for $j = 2, \dots, l+1$,} \\
\frac{2(1 - e^{-18})}{3} e^{x_0 + j-l -2} \epsilon^2 u_j^2 & \mbox{for $j = l+2, \dots, J+1$,} \end{array} \right. 
\end{equation}
where $(u_j)$ is a sequence with $\sum_{j=l+2}^{J+1} u_j^2 \leq 1$ to be specified  later.  Applying Lemma~\ref{triv} with $\G := \{\exp(\tilde{\phi}_a) : \exp(\phi) \in \F(f_0, \delta)\}$, we obtain
\begin{equation}\label{trc}
  H_{[]} \bigl(2^{1/2}\epsilon, \F(f_0, \delta), \shel, [0, a)\bigr) \leq \sum_{j=1}^{J+1}  H_{[]}(\epsilon_j, \G, \shel, S_j).
\end{equation}
We now break the right-hand side of~\eqref{trc} into the three parts:  
\begin{equation*}
  H_1 := N_{[]}(\epsilon_1, \G, \shel, S_1),~  H_2 := \sum_{j=2}^{l+1}
  N_{[]}(\epsilon_j, \G, \shel, S_j) \text{ and }  H_3 :=
  \sum_{j=l+2}^{J+1} N_{[]}(\epsilon_j, \G, \shel, S_j)  
\end{equation*}
and bound each of them below separately. For $H_1$, we have for every $f = e^\phi \in \F(f_0, \delta)$ that 
\begin{equation*}
\delta^2 \geq \int_0^1 (f^{1/2} - f_0^{1/2})^2  = \int_0^1 (e^{\tilde{\phi}_a(x)/2} - 1)^2 \frac{e^{-x}}{1 - e^{-a}} \, dx \geq  \frac{e^{-1}}{1 - e^{-a}}\int_0^1 (e^{\tilde{\phi}_a(x)/2} - 1)^2 \, dx. 
\end{equation*}
Thus, arguing as for the uniform case, since $\delta e^{1/2} (1 - e^{-a})^{1/2} \leq  \kappa e^{1/2} \leq 2^{-5/2}$, 
\begin{equation*}
  H_1 \leq C \log^{5/4} \biggl(\frac{1}{\delta e^{1/2}(1 - e^{-a})^{1/2}} \biggr) \frac{\delta^{1/2}}{\epsilon^{1/2}} \leq C \log^{5/4} 
    \Bigl(\frac{1}{\delta}\Bigr)\frac{\delta^{1/2}}{\epsilon^{1/2}}. 
\end{equation*}
We next bound $H_2$. Note that $\cup_{j=2}^{l+1} S_j \subseteq [1,x_0]$.  We can  therefore apply Lemma \ref{blan} to deduce that whenever $e^\phi \in \mathcal{F}(f_0,\delta)$ and $x \in S_j$, 
\begin{equation*}
  |\tilde{\phi}_a(x)| \leq C e^{x/2}(1 - e^{-a})^{1/2}\delta \leq C e^{j/2}\delta. 
\end{equation*}
An application of Proposition~\ref{mb1} in Section~\ref{Sec:Auxment} therefore gives, for $j = 2,\ldots,l+1$, that 
\begin{equation*}
  H_{[]}(\epsilon_j, \G, \shel, S_j) \leq \frac{C e^{j/4}\delta^{1/2}}{\epsilon_j^{1/2}} \exp(C e^{j/2}\delta) \leq C l^{1/4}
  \frac{\delta^{1/2}}{\epsilon^{1/2}} \exp(C e^{j/2}\delta) \leq C l^{1/4}\frac{\delta^{1/2}}{\epsilon^{1/2}},
\end{equation*}
where the final inequality follows because  
\begin{equation*}
  e^{j/2} \leq e^{1/2} e^{x_0/2} \leq \biggl\{\frac{1}{2^6 (1 - e^{-18})} \biggr\}^{1/2}  \delta^{-1} . 
\end{equation*}
We therefore obtain 
\begin{align*}
  H_2 = \sum_{j=2}^{l+1} H_{[]}(\epsilon_j, \G, \shel, S_j) \leq C l^{5/4} \frac{\delta^{1/2}}{\epsilon^{1/2}} \leq C x_0^{5/4} \frac{\delta^{1/2}}{\epsilon^{1/2}} &\leq C \log^{5/4}\biggl(\frac{1}{2^6 e\delta^2 (1 - e^{-18})}\biggr)\frac{\delta^{1/2}}{\epsilon^{1/2}} \\
  &\leq C\log^{5/4}\Bigl(\frac{1}{\delta}\Bigr)\frac{\delta^{1/2}}{\epsilon^{1/2}}. 
\end{align*}
We next turn to $H_3$, where we consider two cases.  First suppose that $x_0 = a-1$ so that $J = l+2$ and $S_{l+2} = [a-1, a]$. This means that $H_3 = N_{[]}(\epsilon_{l+2}, \G,\shel, [a-1, a])$. We take  $u_{l+2} = 1$ in the definition of  $\epsilon_{l+2}$ in~\eqref{mise}. From the definition of $x_0$ in~\eqref{x0}, we  find that 
\begin{equation}\label{aw}
  a = 1 + x_0 \leq 1 + \log \frac{1}{2^6e\delta^2(1 - e^{-a})} = \log   \frac{1}{2^6\delta^2(1 - e^{-a})}.
\end{equation}
For every $f = e^\phi \in \F(f_0, \delta)$, we can write
\begin{equation*}
  \delta^2 \geq \int_{a-1}^a (f^{1/2} - f_0^{1/2})^2 = \int_{a-1}^a (e^{\tilde{\phi}_a(x)/2} - 1)^2 \frac{e^{-x}}{1 - e^{-a}} \, dx \geq \frac{e^{-a}}{1 - e^{-a}}\int_{a-1}^a (e^{\tilde{\phi}_a/2} - 1)^2. 
\end{equation*}
Now $\delta e^{a/2}(1-e^{-a})^{1/2} \leq 2^{-5/2}$ from~\eqref{aw}, and it follows again by the same argument as in the uniform case that 
\begin{align*}
  H_3 \leq C \log^{5/4} \biggl(\frac{1}{\delta e^{a/2}(1 - e^{-a})^{1/2}}\biggr)\frac{\delta^{1/2} e^{a/4}(1 - e^{-a})^{1/4}}{\epsilon_{l+2}^{1/2}} \leq  C \log^{5/4} \Bigl(\frac{1}{\delta}\Bigr)\frac{\delta^{1/2}}{\epsilon^{1/2}}. 
\end{align*}
Now suppose that $x_0 < a-1$, so that  
\begin{equation*}
  x_0 = \log \frac{1}{2^6e \delta^2 (1 - e^{-a})}. 
\end{equation*}
For every $j \in \{l+2,\ldots,J+1\}$, every $x \in S_j$ and every $f=e^\phi \in \mathcal{F}(f_0,\delta)$, it follows from Lemma~\ref{blan} that  
\begin{equation*}
  \tilde{\phi}_a(x) \leq \frac{8(x - x_0)}{x_0-1} + 7 \leq \frac{8(j-l-1)}{x_0-1} + 7  \leq \frac{8(j - x_0)}{x_0-1} + 7
\end{equation*}
Let $u_j := ce^{-(j-x_0)/4}$  in~\eqref{mise}, where the universal constant $c > 0$ is chosen such that $\sum_{j=l+2}^\infty u_j^2 \leq 1$.  Then by Proposition~\ref{meni}, for $j=l+2,\ldots,J+1$,
\begin{align*}
H_{[]}(\epsilon_j, \G, \shel, S_j) &\leq C\biggl\{1 + \frac{(j-x_0)^{1/2}}{x_0^{1/2}}\biggr\}\frac{e^{2(j-x_0)/(x_0-1)}}{\epsilon_j^{1/2}} \\
&\leq \frac{C}{\epsilon^{1/2}}\biggl\{1 + \frac{(j-x_0)^{1/2}}{x_0^{1/2}}\biggr\}\exp\biggl\{\frac{2(j-x_0)}{x_0-1} - \frac{x_0 + j-l-2}{4} + \frac{j-x_0}{8}\biggr\} \\
&\leq \frac{C}{\epsilon^{1/2}}\biggl\{1 + \frac{(j-x_0)^{1/2}}{x_0^{1/2}}\biggr\}\exp\biggl\{\frac{2(j-x_0)}{x_0-1} - \frac{j-1}{4} + \frac{j-x_0}{8}\biggr\}.
\end{align*}
Hence
\begin{align*}
H_3 &\leq \frac{C}{\epsilon^{1/2}} \sum_{j=l+2}^{J+1} \biggl\{1 +  \frac{(j-x_0)^{1/2}}{x_0^{1/2}}\biggr\}\exp\biggl\{\frac{2(j-x_0)}{x_0-1} - \frac{j-1}{4} + \frac{j-x_0}{8}\biggr\} \\
&\leq \frac{Ce^{-x_0/4}}{\epsilon^{1/2}}\sum_{j=1}^\infty \biggl(1 + \frac{j^{1/2}}{x_0^{1/2}}\biggr)\exp\biggl\{-j\Bigl(\frac{1}{8} - \frac{2}{x_0-1}\Bigr)\biggr\} \leq \frac{Ce^{-x_0/4}}{\epsilon^{1/2}} \leq C\frac{\delta^{1/2}}{\epsilon^{1/2}}. 
\end{align*}
This completes the proof of Theorem~\ref{ment} in the case where $\upsilon = 0$.  We can now treat the case of general $\upsilon \in [0,2^{1/2}]$ as follows.  Fix $f_0 \in \F$, $\epsilon > 0$, let $\kappa$ be as above and let $\delta \in (0,\kappa - \upsilon)$.  Also let $\eta \in (0,\kappa - \upsilon - \delta)$ and $f_1 \in \F^1$ be such that $f_0 \ll f_1$ and 
\begin{equation}\label{ki}
  \shel(f_0, f_1) \leq \upsilon + \eta < \kappa - \delta. 
\end{equation}
Then by the triangle inequality, $\F(f_0, \delta) \subseteq \F\bigl(f_1, \shel(f_0, f_1) + \delta\bigr)$, so that the result in the case $\upsilon = 0$ and~\eqref{ki} give 
\begin{equation*}
  H_{[]}(2^{1/2}\epsilon, \F(f_0, \delta), \shel) \leq C \log^{5/4}\Bigl(\frac{1}{\delta}\Bigr)\frac{\{\delta + \shel(f_0,f_1)\}^{1/2}}{\epsilon^{1/2}} \leq C \log^{5/4}\Bigl(\frac{1}{\delta}\Bigr)\frac{(\delta + \upsilon + \eta)^{1/2}}{\epsilon^{1/2}}.
\end{equation*}
Since $\eta \in (0,\kappa - \upsilon - \delta)$ was arbitrary, the result follows.
\end{proof}

\begin{proof}[Proof of Theorem~\ref{k1}] 
For $\xi \geq 0$ and $\eta \in (0,1)$, let
\[
\tilde{\mathcal{F}}^{\xi, \eta} := \{f \in \mathcal{F}:|\mu_f| \leq \xi,|\sigma_f^2-1| \leq \eta\}.
\]
Since the log-concave maximum likelihood estimator is affine equivariant \citep[Remark 2.3]{DSS2011} and the Hellinger distance between densities is affine invariant, we may assume without loss of generality that $f_0 \in \mathcal{F}^{0,1}$.  By \citet[][Lemma~6]{KimSamworth2016}, there exists a universal constant $\eta \in (0,1)$ such that   
\begin{equation}
\label{Eq:Tail}
\sup_{f_0 \in \mathcal{F}^{0,1}}\mathbb{P}_{f_0} \bigl(\hat f_n \notin \tilde{\mathcal{F}}^{1,\eta} \bigr) \leq \frac{C}{n}.
\end{equation}
For notational convenience, we write 
\[
\upsilon :=   \inf_{f_1 \in \mathcal{F}^1:f_0 \ll f_1} d_{\mathrm{H}}(f_0,f_1),
\]
and initially consider the case $\upsilon \leq \kappa/2$, where $\kappa$ is taken from Theorem~\ref{ment}.  From Theorem~\ref{ment}, we find that
\[
\int_0^\delta H_{[]}^{1/2}\bigl(\epsilon, \mathcal{F}(f_0, \delta) \cap \tilde{F}^{1,\eta},d_{\mathrm{H}}\bigr) \, d\epsilon \leq C\delta^{3/4}(\delta+\upsilon)^{1/4}\log^{5/8}(1/\delta),
\]
provided $\delta \leq \kappa - \upsilon$.  For $\delta > \kappa - \upsilon$, we have
\[
H_{[]}\bigl(\epsilon, \mathcal{F}(f_0, \delta) \cap \tilde{F}^{1,\eta},d_{\mathrm{H}}\bigr) \leq  H_{[]}(\epsilon, \tilde{\mathcal{F}}^{1,\eta},d_{\mathrm{H}}) \leq C\epsilon^{-1/2} \leq C\Bigl(\frac{\delta}{\kappa-\upsilon}\Bigr)^{1/2}\epsilon^{-1/2} \leq \frac{2^{1/2}C}{\kappa^{1/2}}\Bigl(\frac{\delta}{\epsilon}\Bigr)^{1/2},
\]
where the second inequality follows by \citet[][Theorem~4]{KimSamworth2016}.  Thus, in this case,
\[
\int_0^\delta H_{[]}^{1/2}\bigl(\epsilon, \mathcal{F}(f_0, \delta) \cap \tilde{F}^{1,\eta},d_{\mathrm{H}}\bigr) \, d\epsilon \leq C\kappa^{-1/4}\delta.
\]
We can therefore define
\[
\Psi(\delta) := \left\{ \begin{array}{ll} C\delta^{3/4}(\delta+\upsilon)^{1/4}\log^{5/8}(1/\delta) & \mbox{if $\delta \leq \kappa - \upsilon$} \\
C'\kappa^{-1/4}\delta & \mbox{if $\delta > \kappa - \upsilon$,} \end{array} \right.
\]
where the universal constants $C,C' > 0$ are chosen such that 
\[
\Psi(\delta) \geq \max\biggl\{\int_0^\delta H_{[]}^{1/2}(\epsilon, \mathcal{F}(f_0, \delta) \cap \tilde{F}^{1,\eta},d_{\mathrm{H}}) \, d\epsilon \, , \, \delta\biggr\},
\]
and such that $\delta \mapsto \delta^{-2}\Psi(\delta)$ is decreasing on $(0,\infty)$.  Moreover, we can define $\delta_* := \bigl(c n^{-1}\log^{5/4} n + \upsilon^2\bigr)^{1/2}$ for some universal constant $c > 0$, so that for $\delta \geq \delta_*$, we have
\[
\inf_{\delta \geq \delta_*} \frac{n^{1/2}\delta^2}{\Psi(\delta)} \geq \frac{n^{1/2}\delta_*^2}{\Psi(\delta_*)} \geq \frac{n^{1/2}\delta_*}{\max(2^{1/4}C,C'\kappa^{-1/4})\max\{1,\log^{5/8}(1/\delta_*)\}} \geq c'
\]
for some universal constant $c' > 0$.  Then by the empirical process bound of \citet[][Corollary~7.5]{vandegeerbook} (restated as Theorem~\ref{imee} in Section~\ref{Sec:vdg} for convenience), and~\eqref{Eq:Tail}, we deduce that
\begin{align}
\E_{f_0} \kld(\hat f_n, f_0) &= \int_0^{10\log n} \P\bigl[\{\kld(\hat f_n, f_0) \geq t\} \cap \{\hat{f}_n \in \tilde \F^{1,\eta}\}\bigr] \, dt + 10\log n \ \P(\hat{f}_n \notin \tilde \F^{1,\eta}) \nonumber \\
&\hspace{8cm}+ \int_{10\log n}^{\infty} \P\bigl\{\kld(\hat f_n, f_0) \geq t\bigr\} \,dt \nonumber \\
&\leq \delta_*^2 + C \int_{\delta_*^2}^\infty \exp\Bigl(-\frac{nt}{C^2}\Bigr) \, dt + \frac{10\log n}{n} + \int_{10\log n}^\infty \mathbb{P}_{f_0} \biggl\{ \max_{1\leq i\leq n} \log \frac{\hat f_n(X_i)}{f_0(X_i)} \geq t \biggr\} \,dt \nonumber \\
&\leq \frac{C\log^{5/4} n}{n} + \upsilon^2,
\end{align}
where the final inequality follows from~\eqref{ub1} in the proof of~Lemma~\ref{ist}. 

Now suppose that $\upsilon > \kappa/2$. In that case, a slightly simpler version of the calculation above, relying only on the global entropy bound $H_{[]}(\epsilon,\tilde{\mathcal{F}}^{1,\eta},d_{\mathrm{H}}) \leq C \epsilon^{-1/2}$, yields that $\sup_{f_0 \in \mathcal{F}} \mathbb{E}_{f_0} \kld(\hat f_n,f_0) \leq Cn^{-4/5} \leq \kappa/2$ for large $n$; see also \citet[][Theorem~5]{KimSamworth2016}.  By increasing the universal constant to deal with smaller values of $n$ if necessary, the result follows. 
\end{proof}


\section{Appendix}

\subsection{Auxiliary result from Section~\ref{Sec:TV}}

\begin{lemma}\label{coni}Let $g:[a,b] \rightarrow (-\infty,\infty]$ be convex with $g(r) = 0$ for some $r \in [a,b]$. For $\alpha, \beta, c \in \mathbb{R}$, define 
  \begin{equation*}
    G(x) := c + \int_a^x \exp(\alpha t + \beta) g(t) \, dt \qt{for $x \in [a,b]$}. 
  \end{equation*}
Assume $\alpha \neq 0$.  If $r \in (a,b]$, then 
  \begin{equation}\label{coni1}
\inf_{x \in [a,r)} \frac{G(x) - G(r)}{1-e^{-\alpha(r-x)}\{\alpha(r-x)+1\}} = \frac{G(a) - G(r)}{1-e^{-\alpha(r-a)}\{\alpha(r-a)+1\}}     
  \end{equation}
and if $r \in [a,b)$
\begin{equation}\label{coni2}
\sup_{x \in (r,b]} \frac{G(x) - G(r)}{1+e^{\alpha(x-r)}\{\alpha(x-r)-1\}} =  \frac{G(b) - G(r)}{1+e^{\alpha(b-r)}\{\alpha(b-r)-1\}}.  
\end{equation}
Now assume $\alpha = 0$.  If $r \in (a,b]$, then
\begin{equation}\label{coni3}
  \inf_{x \in [a,r)} \frac{G(x) - G(r)}{(r - x)^2}  = \frac{G(a) - G(r)}{(r - a)^2} 
\end{equation}
and if $r \in [a,b)$, then
\begin{equation}\label{coni4}
  \sup_{x \in (r,b]} \frac{G(x) - G(r)}{(x - r)^2} = \frac{G(b) - G(r)}{(b - r)^2}. 
\end{equation}
\end{lemma}
\begin{proof}
Assume $\alpha \neq 0$ and $r \in (a,b]$ and consider the linear function 
\begin{equation*}
     \bar{g}(x) := \frac{\alpha^2\{G(r)-G(a)\}}{e^{\alpha r + \beta} -  e^{\alpha a + \beta}\{\alpha(r-a)+1\}} (r-x).  
   \end{equation*}
Note here that the denominator does not vanish, because $1 - e^{-y}(1 + y) > 0$ for $y \neq 0$.  Thus 
\begin{equation}
\label{Eq:gbar}
\bar{g}(r) = 0 = g(r) ~~~ \text{  and  } ~~~ \int_a^r \exp(\alpha x + \beta) \bar{g}(x) \, dx = \int_a^r \exp(\alpha x + \beta) g(x) \, dx. 
\end{equation}
Now the function $x \mapsto g(x) - \bar{g}(x)$, which is convex on $[a,r]$ and 0 at $x = r$, can change sign at most once in the interval $[a,r)$.  But we deduce from the second part of~\eqref{Eq:gbar} that either this function is zero for all $x \in (a,r]$ or it changes sign exactly once in $(a,r)$.  In particular, there exists $x_0 \in (a, r)$ such that $g(x) \geq \bar{g}(x)$ for $x \in [a,x_0]$ and $g(x) \leq \bar{g}(x)$ for $x \in [x_0,r]$.  This further implies that 
    \begin{equation*}
      \int_a^x \exp(\alpha t + \beta) \{g(t) - \bar{g}(t)\} \, dt = - \int_x^r \exp(\alpha t + \beta) \{g(t) - \bar{g}(t)\} \, dt \geq 0 \qt{for every $x \in [a, r]$}. 
    \end{equation*}
    Consequently, for $x \in [a, r)$, 
 \begin{align*}
 G(x) &= G(r) - \int_x^r \exp(\alpha t + \beta) g(t) \, dt \\
&\geq
 G(r) - \int_x^r  \exp(\alpha t + \beta) \bar g(t) \, dt = G(r) - \frac{1-e^{-\alpha(r-x)}\{\alpha(r-x)+1\}}{1-e^{-\alpha(r-a)}\{\alpha(r-a)+1\}} 
 \{G(r) - G(a)\}. 
\end{align*}
This yields~\eqref{coni1}, and the proof of~\eqref{coni2} is very similar.  The proofs of~\eqref{coni3} and~\eqref{coni4} then follow by taking limits as $\alpha \rightarrow 0$ and using the fact that 
  \begin{equation*}
    \lim_{\alpha \rightarrow 0} \frac{1 - e^{-\alpha y} (\alpha y + 1)}{\alpha^2} = \frac{y^2}{2} \qt{for every $y \in \R$}. 
  \end{equation*}
\end{proof}

\subsection{Auxiliary results from Section~\ref{Sec:kaffine}}

\subsubsection{Auxiliary results for the proof of Theorem~\ref{fik}}

\begin{lemma}\label{ist}
There exists a universal constant $C > 0$ such that for every $n \geq 2$, we have
\begin{equation}\label{ist.eq}
\sup_{f_0 \in \mathcal{F}} \E_{f_0} \biggl\{\sup_{x \in \mathbb{R}} \log \hat{f}_n(x) + \sup_{x \in [X_{(1)},X_{(n)}]} \log \frac{1}{f_0(x)}\biggr\} \leq C \log n.
\end{equation}
\end{lemma}
\begin{proof}
By the affine equivariance of the log-concave maximum likelihood estimator, there is no loss of generality in assuming that $\mu_{f_0}=0$ and $\sigma_{f_0}^2=1$.  Let $\mathcal{P}$ denote the class of probability distributions $P$ on $\mathbb{R}$ for which $\int_{-\infty}^\infty |x| \, dP(x) < \infty$ and $P$ is not a Dirac point mass.  We recall from \citet[][Theorem~2.2]{DSS2011} that there is a well-defined projection $\psi^*:\mathcal{P} \rightarrow \mathcal{F}$ given by
\[
\psi^*(P) := \argmax_{f \in \mathcal{F}} \int_{-\infty}^\infty \log f \, dP.
\]
Now, for $\sigma > 0$, let $\mathcal{P}^{\geq\sigma}$ denote the subset of $\mathcal{P}$ consisting distributions $P$ on the real line with
$\int_{-\infty}^\infty (x - \mu_P)^2 \, dP(x) \geq \sigma^2$, where $\mu_P := \int_{-\infty}^\infty x \, dP(x)$.  By a very similar
argument to that given in the proof of Lemma~6 of \citet{KimSamworth2016},
\[
\sup_{P \in \mathcal{P}^{\geq\sigma}} \sup_{x \in \mathbb{R}}~\psi^*(P)(x) \leq \frac{C}{\sigma}. 
\]
Since $\hat{f}_n = \psi^*(\mathbb{P}_n)$, where $\mathbb{P}_n$ denotes the empirical distribution of $X_1,\ldots,X_n$, we have for $t > 0$ that  
\[
\mathbb{P}\biggl(\sup_{x \in \mathbb{R}} \log \hat{f}_n(x) > \frac{t}{2} \log n\biggr) \leq \mathbb{P}\biggl(\frac{1}{n}\sum_{i=1}^n (X_i - \bar{X})^2 < \frac{C}{n^{t/2}}\biggr) \leq \mathbb{P}\biggl(|X_1 - \bar{X}| < \frac{C^{1/2}}{n^{t/4 - 1/2}}\biggr),
\]
where $\bar{X} := n^{-1}\sum_{i=1}^n X_i$.  But $X_1 - \bar{X}$ has mean 0, variance $1 - 1/n$ and has a log-concave density (which is therefore bounded by a universal constant).  Hence
\begin{equation}\label{boundSup}
\mathbb{P}\biggl(\sup_{x \in \mathbb{R}} \log \hat{f}_n(x) > \frac{t}{2} \log n\biggr) \leq \frac{C}{n^{t/4-1/2}}.
\end{equation}
Now write $X_{(1)} := \min_{1 \leq i \leq n} X_i$ and $X_{(n)} := \max_{1 \leq i \leq n} X_i$, let $F_0$ denote the distribution function corresponding to $f_0$ and for $t \geq 2$ let
\[
\Omega_t := \{X_{(1)} \geq F_0^{-1}(n^{-t/2}/\alpha)\} \cap \{X_{(n)} \leq F_0^{-1}(1-n^{-t/2}/\alpha)\},
\]
where $\alpha > 0$ is taken from Lemma~\ref{Lemma:Bobkov} below.  Then by a union bound,
\begin{equation}
\label{Eq:Union}
\sup_{f_0 \in \mathcal{F}^{0,1}} \mathbb{P}_{f_0}(\Omega_t^c) \leq \frac{2}{\alpha n^{t/2-1}}.
\end{equation}
Moreover, on $\Omega_t$,
\begin{align}
\label{Eq:ConvHull}
\sup_{x \in [X_{(1)},X_{(n)}]} \log \frac{1}{f_0(x)} &\leq \sup_{x \in [F_0^{-1}(n^{-t/2}/\alpha),F_0^{-1}(1-n^{-t/2}/\alpha)]} \log \frac{1}{f_0(x)} \nonumber \\
&= \max\biggl\{\log \frac{1}{f_0\bigl(F_0^{-1}(n^{-t/2}/\alpha)\bigr)} \, , \, \log \frac{1}{f_0\bigl(F_0^{-1}(1 - n^{-t/2}/\alpha)\bigr)}\biggr\} \nonumber \\
&\leq \frac{t}{2} \log n,
\end{align}
where the equality holds because the minimum of a concave function on a compact interval is attained at one of the endpoints of the interval, and the second inequality holds due to Lemma~\ref{Lemma:Bobkov} below.  It follows from~\eqref{boundSup},~\eqref{Eq:Union} and~\eqref{Eq:ConvHull} that for $t \geq 2$,
\begin{align}
\label{ub1}
\mathbb{P}\biggl(\sup_{x \in \mathbb{R}} \log \hat{f}_n(x) + \sup_{x \in [X_{(1)},X_{(n)}]} \log \frac{1}{f_0(x)} > t \log n\biggr) &\leq\mathbb{P}\biggl(\sup_{x \in \mathbb{R}} \log \hat{f}_n(x) > \frac{t}{2} \log n\biggr) \nonumber \\
&\hspace{0.5cm}+ \mathbb{P}\biggl(\sup_{x \in [X_{(1)},X_{(n)}]} \log \frac{1}{f_0(x)} > \frac{t}{2} \log n\biggr) \nonumber \\
&\leq \frac{C}{n^{t/4-1/2}} + \frac{2}{\alpha n^{t/2-1}},
\end{align}
and the result follows.
\end{proof}
The following result is a small generalisation of Proposition~A.1(c)
of \citet{Bobkov1996}.  
\begin{lemma}
\label{Lemma:Bobkov}
There exists $\alpha > 0$ such that for all $p \in (0,1)$,
\[
\inf_{f_0 \in \mathcal{F}^{0,1}} f_0\bigl(F_0^{-1}(p)\bigr) \geq \alpha \min(p,1-p).
\]
\end{lemma}
\begin{proof}
Proposition~A.1(c) of \citet{Bobkov1996} gives that $p \mapsto f_0\bigl(F_0^{-1}(p)\bigr)$ is positive and concave on $(0,1)$.  But, by Theorem~5(b) of \citet{KimSamworth2016}, there exists $\alpha > 0$ such that 
\[
\inf_{f_0 \in \mathcal{F}^{0,1}} f_0(0) \geq \alpha.
\]
Noting that $F_0(0) \in (0,1)$, we deduce by concavity that for $p \in (0,F_0(0)]$,
\[
\inf_{f_0 \in \mathcal{F}^{0,1}} f_0\bigl(F_0^{-1}(p)\bigr) \geq \frac{p}{F_0(0)}\alpha \geq \alpha p \geq \alpha\min(p,1-p).
\] 
A very similar argument handles the case $p \in \bigl(F_0(0),1)$, and this concludes the proof.
\end{proof}

\subsubsection{Auxiliary results for the proof of Theorem~\ref{ment}}
\label{Sec:Auxment}

Recall that we can write $\mathcal{F}^1 = \{f_{\alpha,s_1,s_2}:(\alpha,s_1,s_2) \in \mathcal{T}\}$.
\begin{lemma}\label{affi}
If $X \sim f_{\alpha,s_1,s_2} \in \mathcal{F}^1$, then there exist $a \neq 0$ and $b \in \mathbb{R}$ such that $aX+b$ has a density $f_0 \in \mathcal{F}^1$ of one of the following three forms:
\begin{enumerate}
\item $f_0 = f_{0,0,1}$;
\item $f_0 = f_{-\alpha_0,0,1}$ for some $\alpha_0 \in (0,18)$;
\item $f_0 = f_{-1,0,s_0}$ for some $s_0 \in [18,\infty]$.   
  \end{enumerate}
\end{lemma}
\begin{proof}
Let $X \sim f_{\alpha,s_1,s_2} \in \mathcal{F}^1$ for some $(\alpha,s_1,s_2) \in \mathcal{T}$, and let $a \neq 0$ and $b \in \mathbb{R}$.  Then
\[
aX+b \sim \left\{ \begin{array}{ll} f_{\alpha/a,as_1+b,as_2+b} & \mbox{if $a > 0$} \\
f_{\alpha/a,as_2+b,as_1+b} & \mbox{if $a < 0$.} \end{array} \right.
\]
Thus, if $\alpha = 0$, we can set $a = (s_2-s_1)^{-1}$, $b = -s_1(s_2-s_1)^{-1}$ so that $aX+b \sim f_{0,0,1}$.  If $\alpha > 0$ and $\alpha(s_2-s_1) < 18$, then we can set $a = -(s_2-s_1)^{-1}$, $b = s_2(s_2-s_1)^{-1}$ while if $\alpha < 0$ and $|\alpha|(s_2-s_1) < 18$ then we can set $a = (s_2-s_1)^{-1}$, $b = -s_1(s_2-s_1)^{-1}$; in either situation, $aX+b \sim f_{-\alpha_0,0,1}$, with $\alpha_0 := |\alpha|(s_2-s_1) \in (0,18)$.  Finally, if $\alpha > 0$ and $\alpha(s_2-s_1) \in [18,\infty]$, then we can set $a = -\alpha$, $b = \alpha s_2$ while if $\alpha < 0$ and $|\alpha|(s_2-s_1) \in [18,\infty]$ then we can set $a = -\alpha$, $b=\alpha s_1$; in either situation, $aX+b \sim f_{-1,0,s_0}$ with $s_0 := |\alpha|(s_2-s_1)$.
\end{proof}

\begin{lemma}\label{ub}
Let $\phi : \R \rightarrow [-\infty, \infty)$ be a concave function whose domain is contained in $[0,1]$ and which satisfies   
\begin{equation}\label{lb.con}
\int_0^1 (e^{\phi(u)/2}  - 1)^2 \, du \leq \delta^2 
\end{equation}
for some $\delta \in (0,2^{-5/2}]$. Then 
\begin{equation}\label{ub1.eq}
  \phi(x) \leq 2^{13/2} \delta \qt{for every $x \in [0,1]$.}
\end{equation}
Moreover,
\begin{equation}\label{ub2.eq}
  \phi(x) \geq \frac{-4\delta}{\{\min(x, 1-x)\}^{1/2}} \qt{when $\min(x, 1-x) \geq 4 \delta^2$}. 
\end{equation}
\end{lemma}
\begin{proof}
We first prove inequality \eqref{ub1.eq}. By symmetry, it suffices to prove that $\phi(x) \leq 2^{13/2}\delta$ for all $x \in [0,1/2]$. Fix $x \in [0,1/2]$ and assume that $\phi(x) > 0$, for otherwise there is nothing to prove. Let $x_* \in (x, 1]$ be such that $\phi(x_*) = 0$ if such an $x_*$ exists; otherwise, set $x_* = 1$. 

We first consider the case $x_* \geq 3/4$.  Since $e^x \geq 1+x$ and $\phi$ is a concave function with $\phi(x_*) \geq 0$, 
\begin{align*}
\delta^2 \geq \int_x^{x_*} (e^{\phi(u)/2}  - 1)^2 \, du \geq \frac{1}{4} \int_x^{x_*} \phi^2(u) \, du \geq \frac{\phi^2(x)}{4}\int_x^{x_*} \biggl(\frac{x_* - u}{x_*- x}\biggr)^2 \, du &= \frac{x_*-x}{12}\phi^2(x) \\
&\geq \frac{\phi^2(x)}{48},
\end{align*}
so $\phi(x) \leq 4\sqrt{3}\delta$.

Now suppose instead that $x_* < 3/4$, so that $\phi(x_*) = 0$.  Then for $u \in [7/8,1]$,
\[
\phi(u) \leq -\frac{u-x_*}{x_* - x} \phi(x) \leq - \frac{\phi(x)}{8}.
\]
We deduce that
\[
\delta^2 \geq \int_{7/8}^1 (1 - e^{\phi(u)/2})^2 \, du \geq \frac{1}{8}(1 - e^{-\phi(x)/16})^2,
\]
so 
\[
\phi(x) \leq 16\log\biggl(\frac{1}{1 - 2^{3/2}\delta}\biggr) \leq \frac{2^{11/2}\delta}{1-2^{3/2}\delta} \leq 2^{13/2}\delta,
\]  
since $\delta \in (0,2^{-5/2}]$.  This completes the proof of \eqref{ub1.eq}. 

We now proceed to prove inequality \eqref{ub2.eq}, and by symmetry it suffices to consider a fixed $x \in [4\delta^2,1/2]$.  We assume that $\phi(x) < 0$, because otherwise there is nothing to prove.  By concavity of $\phi$, we have either $\phi(u) \leq \phi(x)$ for all $u \in [0,x]$ or $\phi(u) \leq \phi(x)$ for all $u \in [x,1]$.  In the former case, 
\[
\delta^2 \geq \int_0^{x} (1 - e^{\phi(u)/2})^2 \, du \geq x(1 - e^{\phi(x)/2})^2.
\]
Thus
\[
\phi(x) \geq 2\log\biggl(1 - \frac{\delta}{x^{1/2}}\biggr) \geq \frac{-4\delta}{x^{1/2}}.
\]
In the latter case, where $\phi(u) \leq \phi(x)$ for all $u \in [x,1]$, we find
\[
\delta^2 \geq \int_x^1 (1 - e^{\phi(u)/2})^2 \, du \geq (1-x)(1 - e^{\phi(x)/2})^2 \geq x(1 - e^{\phi(x)/2})^2,
\]
and the conclusion follows as before.
\end{proof}

\begin{lemma}\label{blan}
Let $f_0 = f_{-1,0,a} \in \mathcal{F}^1$ for some $a \in [18,\infty]$, and let ${\phi} : \R \rightarrow [-\infty, \infty)$ be a concave function whose domain is contained in $[0, a]$ and which satisfies   
\begin{equation}\label{blan.con}
    \int_0^a \{e^{\phi(u)/2} - f_0^{1/2}(u)\}^2 \, du \leq \delta^2   
\end{equation}
for some $\delta \in (0,e^{-9}/8]$.  Let 
  \begin{equation}\label{x0}
    x_0 := \min \biggl\{\log \frac{1}{2^6 e \delta^2 (1 - e^{-a})} \, , \, a-1\biggr\} \geq 17. 
  \end{equation}
Then with $\tilde{\phi}_a$ defined as in~\eqref{Eq:phitildea}, we have
\begin{equation} \label{blan.lb}
-4\frac{e^{x/2}(1 - e^{-a})^{1/2}}{(1 - e^{-1})^{1/2}}\delta \leq \tilde{\phi}_a(x) \leq 2^{13/2} e^{x/2}(1 - e^{-a})^{1/2}\delta
\qt{for every $x \in [1,x_0]$}
\end{equation}
and 
\begin{equation}\label{blan.ub2}
  \tilde{\phi}_a(x) \leq 8 \frac{x - x_0}{x_0 - 1} + 7 \qt{for every $x \in [x_0,a]$}.  
\end{equation}
\end{lemma}
\begin{proof}
Fix $f_0 = f_{-1,0,a}$ for some $a \in [18,\infty]$, $\delta \in (0,e^{-9}/8]$ and $\phi$ satisfying the conditions of the lemma.  
For ease of notation, let us denote $\tilde{\phi}_a$ by $\psi$. 
We first prove the lower bound for $\psi$ in \eqref{blan.lb}. Fix $x \in [1,x_0]$ and assume that $\psi(x) < 0$ because otherwise there is nothing to prove.  By concavity of $\psi$, the inequality $\psi(u) \leq \psi(x)$ is true either for all $u \in [0,x]$ or for all $u \in [x,a]$. In the former case, 
\begin{align}
\label{Eq:lb1}
\delta^2 &\geq \int_0^x \{e^{\phi(u)/2} - f_0^{1/2}(u)\}^2 \, du = \int_0^x (1-e^{\psi(u)/2})^2 \frac{e^{-u}}{1 -
      e^{-a}} \, du  \nonumber \\
&\geq (1-e^{\psi(x)/2})^2\frac{1 - e^{-x}}{1 - e^{-a}} \geq (1-e^{\psi(x)/2})^2\frac{e^{-x}(e - 1)}{1 - e^{-a}},
\end{align}
where we used the fact that $x \geq 1$ in the final inequality.  Similarly in the latter case, we can consider the integral from $x$ to $a$ instead to obtain
\begin{equation}
\label{Eq:lb2}
\delta^2 \geq (1-e^{\psi(x)/2})^2\frac{e^{-x} - e^{-a}}{1 - e^{-a}} \geq (1-e^{\psi(x)/2})^2\frac{e^{-x} (1 - e^{-1})}{1-e^{-a}},
\end{equation}
where we used the fact that $x \leq a-1$ for the final inequality.  Now
\[
\frac{e^{x/2}(1 - e^{-a})^{1/2}}{(1 - e^{-1})^{1/2}}\delta  \leq \frac{e^{x_0/2}(1 - e^{-a})^{1/2}}{(1 - e^{-1})^{1/2}}\delta \leq \frac{1}{2},
\]
and we deduce from~\eqref{Eq:lb1} and~\eqref{Eq:lb2} that 
\[
\psi(x) \geq 2\log\biggl(1 - \frac{e^{x/2}(1 - e^{-a})^{1/2}}{(1 - e^{-1})^{1/2}}\delta\biggr) \geq -\frac{4e^{x/2}(1 - e^{-a})^{1/2}}{(1 - e^{-1})^{1/2}}\delta,
\]
as required.

We next prove the upper bound in \eqref{blan.lb}. To this end, again fix $x \in [1,x_0]$ and note by very similar arguments to those above that  
\begin{equation*}
  \delta^2 \geq \int_{x-1}^x (e^{\psi(u)/2}-1)^2 \frac{e^{-u}}{1 - e^{-a}} \, du \geq \frac{e^{-x}}{1 - e^{-a}}\int_0^1 (e^{\psi(u+x-1)/2}-1)^2 \, du.
\end{equation*}
Now
\[
e^{x/2}(1-e^{-a})^{1/2} \delta \leq e^{x_0/2}(1-e^{-a})^{1/2} \delta \leq \frac{1}{8e^{1/2}} \leq 2^{-5/2},
\]
so the result follows by~\eqref{ub1.eq} in Lemma~\ref{ub}. 

Finally, we prove~\eqref{blan.ub2}. Fix $x \in [x_0,a]$. Inequality \eqref{blan.lb} gives 
\begin{equation*}
  \psi(x_0) \leq 2^{13/2} e^{x_0/2}(1 - e^{-a})^{1/2}\delta \leq 2^{7/2}e^{-1/2}
\end{equation*}
and also that 
\begin{equation*}
  \psi(1) \geq -4\frac{e^{1/2}(1 - e^{-a})^{1/2}}{(1 - e^{-1})^{1/2}}\delta \geq - \frac{e^{1/2}}{2 e^9(1 - e^{-1})^{1/2}} \geq -\frac{1}{2}.
\end{equation*}
It therefore follows by concavity of $\psi$ that 
\begin{equation*}\label{ttp}
  \psi(x) \leq \frac{x - x_0}{x_0 - 1} \{\psi(x_0) - \psi(1)\} + \psi(x_0) \leq 8 \frac{x - x_0}{x_0 - 1} + 7,
\end{equation*}
as required.   
\end{proof}


In order to prove Theorem~\ref{ment} for these three cases, we need to prove two results on the bracketing numbers of log-concave functions on bounded subintervals of $\R$.  For $a < b$ and $-\infty \leq B_1 \leq B_2 < \infty$, let $\F([a, b], B_1, B_2)$ denote the class of all non-negative functions $f$ on $[a, b]$ such that $\log f$ is concave and such that $B_1 \leq \log f(x) \leq B_2$ for every $x \in [a,b]$. 
\begin{proposition}\label{mb1}
There exists a universal constant $C > 0$ such that 
  \begin{equation}\label{mb1.eq}
H_{[]} \bigl(\epsilon, \F([a, b], B_1, B_2), \shel, [a, b] \bigr)  \leq C(B_2-B_1)^{1/2} \frac{e^{B_2/4} (b-a)^{1/4}}{\epsilon^{1/2}} 
  \end{equation}
for every $\epsilon > 0$, $a < b$ and $-\infty \leq B_1 \leq B_2 < \infty$. 
\end{proposition}
\begin{proof}
Fix $\epsilon > 0$, $a < b$ and $B_1 \leq B_2$, and let $\delta := 2 \epsilon e^{-B_2/2}$.  By \citet[][Proposition~4]{KimSamworth2016b} (see also \citet{GuntuboyinaSen,DossWellner}), there exists a bracketing set $\{[\phi_{L,j}, \phi_{U,j}]:j=1,\ldots,M\}$ for the set of concave functions on $[a,b]$ that are bounded below by $B_1$ and above by $B_2$ with $\int_{a}^b (\phi_{U, j} - \phi_{L, j})^2 \, dx \leq \delta^2$ and\footnote{In fact, formally, only the case $B_1 = -B_2$ is covered by \citet[][Proposition~4]{KimSamworth2016b}, but the proof proceeds by first considering the case $B_1=-1$, $B_2=1$, so a simple scaling argument can be used to obtain the claimed result.}
\[
\log M \leq C\biggl\{\frac{(b-a)^{1/2}(B_2-B_1)}{\delta}\biggr\}^{1/2}.      
\]
Now take $f_{L, j} := e^{\phi_{L,j}}$ and $f_{U, j} := e^{\phi_{U,j}}$ for $j = 1,\ldots,M$. Since there is no loss of generality in assuming $\phi_{U,j}(x) \leq B_2$ for every $j \in \{1,\ldots,M\}$ and $x\in [a,b]$, we have
\begin{align*}
\int_a^b (f_{U, j}^{1/2} - f_{L, j}^{1/2})^2   &= \int_a^b e^{\phi_{U,j}} \bigl\{1 - e^{-(\phi_{U,j} - \phi_{L,j})/2}\bigr\}^2 \leq \frac{e^{B_2}}{4} \int_a^b (\phi_{U,j}  - \phi_{L,j})^2 \leq \frac{\delta^2}{4} e^{B_2} = \epsilon^2. 
\end{align*}
The result follows.
\end{proof}
In the case where $B_1 = -\infty$, Proposition~\ref{mb1} unfortunately gives the trivial bound $H_{[]}(\epsilon, \F([a, b], -\infty, B_2), \shel, [a, b]) \leq \infty$. It turns out however that this quantity is actually finite, as shown by the following result.
\begin{proposition}\label{meni}
There exists a universal constant $C > 0$ such that 
  \begin{equation}\label{meni.eq}
H_{[]}\bigl(\epsilon, \F([a, b], -\infty, B), \shel, [a, b]\bigr) \leq C (1 + B^{1/2}) \frac{e^{B/4} (b-a)^{1/4}}{\epsilon^{1/2}}  
  \end{equation}
for every $\epsilon > 0$, $a < b$ and $B \geq 0$.  
\end{proposition}
\begin{proof}
First consider the case $B=0$.  Fix $\epsilon > 0$ and let $k$ be the smallest integer for which $e^{-k} (b-a) \leq \epsilon^2$.  We write $f \in \F([a, b], -\infty, 0)$ as 
\begin{equation*}
  f = \sum_{j=1}^k f \one_{\{-j < \log f \leq -(j-1)\}} + f \one_{\{\log f \leq -k\}}. 
\end{equation*}
For the final term, a single bracket suffices, because $0 \leq f \one_{\{\log f \leq -k\}} \leq e^{-k}$, and $\int_a^b (e^{-k/2} - 0)^2 = e^{-k}(b-a) \leq \epsilon^2$.  Writing $\epsilon_j := \epsilon j e^{-j/4}/8$, so that $\sum_{j =1}^{\infty} \epsilon_j^2 \leq \epsilon^2/2$, it therefore follows by Lemma~\ref{golo} and Proposition~\ref{nny} below that 
\begin{align*}
  H_{[]}\bigl(\epsilon, \F([a, b], -\infty, 0),\shel, [a, b]\bigr) &\leq C(b-a)^{1/4} \sum_{j=1}^k \bigl\{1 +
    (2j - 1)^{1/2}\bigr\} \frac{e^{-(j-1)/4}}{\epsilon_j^{1/2}} \\
 &\leq C (b-a)^{1/4}  \sum_{j=1}^k j^{1/2}e^{-j/4}\epsilon_j^{-1/2} \\
&= C \frac{(b-a)^{1/4}}{\epsilon^{1/2}}\sum_{j=1}^k e^{-j/8} \leq C
 \frac{(b-a)^{1/4}}{\epsilon^{1/2}}. 
\end{align*}
This proves~\eqref{meni.eq} for the case $B = 0$.  For general $B \geq 0$, we write $f \in \F([a, b], -\infty, B)$ as 
\begin{equation*}
  f = f \one_{\{0  < \log f \leq B\}} + f \one_{\{-\infty \leq \log f \leq 0\}}.
\end{equation*}
By Lemma~\ref{golo} and Proposition~\ref{nny} again, it therefore follows that 
\begin{align*}
  H_{[]}\bigl(\epsilon, \F([a, b], -\infty, B), \shel, [a, b]\bigr)  
&\leq C(1 + B^{1/2})\frac{e^{B/4} (b-a)^{1/4}}{\epsilon^{1/2}} + C\frac{(b-a)^{1/4}}{\epsilon^{1/2}} \\
&\leq C(1 + B^{1/2})\frac{e^{B/4} (b-a)^{1/4}}{\epsilon^{1/2}},
\end{align*}
as required.
\end{proof}
We need two more results for the proof of Proposition~\ref{meni}. 
\begin{lemma}\label{golo}
Suppose $\G, \G_1, \dots, \G_k$ are classes of non-negative functions on $S \subseteq \R$ such that  
\begin{equation*}
    \G \subseteq \G_1 + \dots + \G_k. 
\end{equation*}
Then for every $\epsilon, \epsilon_1, \dots, \epsilon_k > 0$ such that $\sum_{j=1}^k \epsilon_j^2 \leq \epsilon^2$, we have  
  \begin{equation*}
    H_{[]}(\epsilon, \G, \shel, S) \leq \sum_{j=1}^k H_{[]}(\epsilon_j,
    \G_j, \shel, S).  
  \end{equation*}
\end{lemma}
\begin{proof}
Fix $\epsilon, \epsilon_1, \dots, \epsilon_k > 0$ such that $\sum_{j=1}^k \epsilon_j^2 \leq \epsilon^2$. For each $j \in \{1,\ldots,k\}$, there exists an $\epsilon_j$-Hellinger bracketing set $\{[g_{L,j,\ell}, g_{U,j,\ell}], \ell = 1,\dots, M_j\}$ for $\mathcal{G}_j$ with $M_j := N_{[]}(\epsilon_j, \G_j, d_{\mathrm{H}})$.  For $\ell = (\ell_1,\ldots,\ell_k) \in \times_{j=1}^k \{1,\ldots,M_j\}$, we now define
\[
h_{L,\ell}(x) := \sum_{j=1}^k g_{L,j,\ell_j}, \quad h_{U,\ell}(x) := \sum_{j=1}^k g_{U,j,\ell_j}.
\]
If $g \in \mathcal{G}$, then $g = \sum_{j=1}^k g_j$ for some $g_j \in \mathcal{G}_j$, so for each $j$ we can find $\ell_j^* \in \{1,\ldots,M_j\}$ such that $g_{L,j,\ell_j^*} \leq g_j \leq g_{U,j,\ell_j^*}$.  Thus, writing $\ell^* = (\ell_1^*,\ldots,\ell_k^*)$, we have $h_{L,\ell^*} \leq g \leq h_{U,\ell^*}$.  Moreover, 
\[
d_{\mathrm{H}}(h_{U,\ell},h_{L,\ell}) = \int_{\mathcal{S}} (h_{U,\ell}^{1/2} - h_{L,\ell}^{1/2})^2 \leq \sum_{j=1}^k \int_{\mathcal{S}} (g_{U,j,\ell_j}^{1/2} - g_{L,j,\ell_j}^{1/2})^2 \leq \sum_{j=1}^k \epsilon_j^2 \leq \epsilon^2,
\]
where in the first inequality we have used the fact that for $a_1,\ldots,a_k,b_1,\ldots,b_k \geq 0$ with $a := \sum_{j=1}^k a_j$, $b := \sum_{j=1}^k b_j$, we have
\[
(a^{1/2}-b^{1/2})^2 \leq \sum_{j=1}^k (a_j^{1/2} - b_j^{1/2})^2,
\]
which can be proved by induction.  The result follows.  
\end{proof}

\begin{lemma}\label{triv}
Let  $S, S_1, S_2, \dots S_k$ denote measurable subsets of $\mathbb{R}$ such that $S \subseteq \cup_{j=1}^{k} S_j$. Let $\F_0$ denote an arbitrary class of non-negative functions on $\cup_{j=1}^{k} S_j$ and let $\G := \{e^{\tilde{\phi}_a}: e^\phi \in \F_0\}$, where $\tilde{\phi}_a$ is defined in~\eqref{Eq:phitildea}.  Let $\alpha_j := \inf \{x : x \in S_j\}$ and suppose that $\epsilon,\epsilon_1,\ldots,\epsilon_k > 0$ satisfy
\[
\sum_{j=1}^{k} e^{-\alpha_j} \epsilon_j^2 \leq (1 - e^{-a})\epsilon^2. 
\]
Then
\begin{equation}\label{triv.eq}
  H_{[]}(\epsilon, \F_0, \shel, S) \leq \sum_{j=1}^k H_{[]}(\epsilon_j,\G, \shel, S_j). 
\end{equation}
\end{lemma}
\begin{proof}
We may assume that $S_1, \dots, S_k$ are pairwise disjoint, because otherwise we can work with the sets $S_1' := S_1$ and $S_j' := S_j \setminus \cup_{\ell=1}^{j-1} S_\ell$ for $j=2,\ldots,k$.  For each $j=1,\ldots,k$, let $\{[f_{L, \ell}^{(j)},f_{U, \ell}^{(j)}]: \ell = 1, \dots, N_{[]}(\epsilon_j,\mathcal{G},d_{\mathrm{H}},S_j)\}$ denote an $\epsilon_j$-Hellinger bracketing set for the class $\G$ over $S_j$.  Now, for $x \in S_j$ and $\ell = (\ell_1,\ldots,\ell_k) \in \bigl\{1,\ldots,N_{[]}(\epsilon_1,\mathcal{G},d_{\mathrm{H}},S_1)\bigr\} \times \ldots \times \bigl\{1,\ldots,N_{[]}(\epsilon_k,\mathcal{G},d_{\mathrm{H}},S_k)\bigr\}$, set 
  \begin{equation*}
    f_{L, \ell}(x) := \frac{e^{-x} f_{L, \ell_j}^{(j)}(x)}{1 - e^{-a}} ~~
    \text{ and } ~~ f_{U, \ell}(x) := \frac{e^{-x} f_{U, \ell_j}^{(j)}(x)}{1 - 
      e^{-a}} . 
  \end{equation*}
Then for every $f \in \F_0$, there exists $\ell = (\ell_1,\ldots,\ell_k)$ such that $f_{L, \ell} \leq f \leq f_{U, \ell}$.  Moreover,
\begin{align*}
 \int_{S} (f_{U, \ell}^{1/2} - f_{L, \ell}^{1/2})^2 &\leq \sum_{j=1}^k \int_{S_j} \frac{e^{-x}}{1 - e^{-a}}\bigl\{f_{U,\ell_j}^{(j)}(x)^{1/2} - f_{L, \ell_j}^{(j)}(x)^{1/2}\bigr\}^2 \, dx \leq \sum_{j=1}^k \frac{e^{-\alpha_j}}{1 - e^{-a}} \epsilon_j^2 \leq \epsilon^2,
  \end{align*}
as required.
\end{proof}
The following result is also used in the proof of Proposition~\ref{meni}. For $a < b$ and $-\infty < B_1 \leq B_2 < \infty$, let $\F'([a, b], B_1, B_2)$ denote the class of all functions on $[a,b]$ of the form $x \mapsto f(x) \one_{\{B_1 \leq \log f(x) \leq B_2\}}$ for some non-negative function $f$ for which $\log f: [a, b] \rightarrow [-\infty, \infty)$ is concave.  
\begin{proposition}\label{nny}
  There exists a universal constant $C > 0$ such that 
  \begin{equation*}
    H_{[]}(\epsilon, \F'([a, b], B_1, B_2), \shel, [a, b]) \leq C \{1 + (B_2-B_1)^{1/2}\} \frac{e^{B_2/4}(b-a)^{1/4}}{\epsilon^{1/2}} 
  \end{equation*}
  for every $\epsilon > 0, a < b$ and $-\infty < B_1 \leq B_2 < \infty$. 
\end{proposition}
\begin{proof}
We initially prove the result for the subclass $\F''([a, b], B_1, B_2) \subseteq \F'([a, b], B_1, B_2)$ consisting of all non-negative functions $f$ on $[a, b]$ such that $\log f$ is concave and such that for every $x \in [a, b]$, either $B_1 \leq \log f(x) \leq B_2$ or $f(x) = 0$.  Assume for now that $\epsilon \in (0,2e^{B_2/2}(b-a)^{1/2}]$.  Let $\eta := \epsilon/2^{1/2}$ and $\delta := \epsilon^2 e^{-B_2}/4$. Let $G$ be a $\delta$-grid in $[a, b]$ with size $|G| \leq 2 + (b-a)/\delta$. We assume that the endpoints $a,b$ are both included in $G$.  By Proposition~\ref{mb1}, for every $g_1,g_2 \in G$ with $g_1 < g_2$, we can construct $\eta$-Hellinger brackets $\{[f_{L, j}, f_{U, j}]: j = 1, \dots, M\}$ for $\F([g_1, g_2],B_1, B_2)$ on the interval $[g_1, g_2]$, with
\begin{equation}\label{zo}
\log M \leq C (B_2-B_1)^{1/2} \frac{e^{B_2/4}(g_2 - g_1)^{1/4}}{\eta^{1/2}} \leq C (B_2-B_1)^{1/2}\frac{e^{B_2/4}(b - a)^{1/4}}{\eta^{1/2}}. 
\end{equation}
If $g_1 = a$, let $g_1' := a$; otherwise let $g_1' \in G$ be the closest point to $g_1$ that is strictly smaller than $g_1$. Similarly, if $g_2 = b$, let $g_2' := b$; otherwise let $g_2' \in G$ be the closest point to $g_2$ that is strictly larger than $g_2$.  Thus $|g_i - g_i'| \leq \delta$ for $i = 1, 2$. We extend the definitions of $f_{U, j}$ and $f_{L,j}$ to the whole of $[a,b]$ by defining  
\[
f_{U, j}(x) := \left\{ 
  \begin{array}{ll} 
   e^{B_2} & \mbox{for $x \in [g_1', g_1) \cup (g_2, g_2']$} \\ 
   0       & \mbox{for $x  \notin [g_1', g_2']$,}
\end{array} \right.
\]
and $f_{L, j}(x) := 0$ for $x \notin [g_1, g_2]$.  Then 
\[
\int_a^b \bigl(f_{U, j}^{1/2} - f_{L, j}^{1/2}\bigr)^2 \leq e^{B_2} (|g_1 - g_1'| + |g_2 - g_2'|) + \eta^2 \leq 2 \delta e^{B_2} + \eta^2 =\epsilon^2. 
\]
We take the union of these brackets over $g_1, g_2 \in G$ with $g_1 < g_2$. 

In addition to these brackets, we also consider, for each $g_1,g_2 \in G$ with $g_1 < g_2$ and $|g_1 - g_2| \leq 4 \delta$, one bracketing pair $(f_L, f_U)$ with $f_L(x) := 0$ for all $x \in [a,b]$ and $f_U(x) := e^{B_2} \one_{\{x \in [g_1,g_2]\}}$. The size of this bracket is given by 
\begin{equation*}
 \int_a^b (f_U^{1/2} - f_L^{1/2})^2 = e^{B_2} |g_1 - g_2| \leq 4 \delta e^{B_2} = \epsilon^2.  
\end{equation*}
We claim that for every $f \in \F''([a, b], B_1, B_2)$, there exists a bracketing pair for which $f_{L, j} \leq f \leq f_{U, j}$.  To see this, fix $f \in \F''([a, b], B_1, B_2)$ with support $[a', b'] \subseteq [a, b]$.  Suppose for now that $|a' - b'| \leq 2 \delta$.  Set $g_1 := \sup\{g \in G: g \leq a'\}$ and $g_2 := \inf\{g \in G: g \geq b'\}$, so that $|g_1 - g_2| \leq 4 \delta$.  Then pair $(f_L, f_U)$ with $f_L(x) = 0$ for all $x \in [a,b]$ and $f_U(x) = e^{B_2}\one_{\{x \in [g_1,g_2]\}}$ satisfies $f_L(x) \leq f(x) \leq f_U(x)$ for all $x \in [a,b]$.
 
On the other hand, now suppose that $|a' - b'| > 2 \delta$. Let $g_1 := \inf\{g \in G: g \geq a'\}$ and $g_1' := \sup\{g \in G: g \leq a'\}$. Similarly let $g_2 := \sup\{g \in G: g \leq b'\}$ and $g_2' := \inf\{g \in G: g \geq b'\}$.  By our construction, there exists a pair $(f_{L, j}, f_{U, j})$ with $f_{L,j}(x) \leq f(x) \leq f_{U,j}(x)$ for all $x \in [g_1,g_2]$.  But then our definition of $f_{L,j},f_{U,j}$ on $[a,b] \setminus [g_1,g_2]$ in fact ensures that $f_{L, j}(x) \leq f(x) \leq f_{U, j}(x)$ for all $x \in [a,b]$.

The cardinality $N$ of our bracketing set satisfies
\[
\log N \leq \log(|G|^2 M + |G|^2) \leq \log \biggl(2 \Bigl(2 + \frac{b-a}{\delta}\Bigr)^2\biggr) + C (B_2-B_1)^{1/2}\frac{e^{B_2/4}(b - a)^{1/4}}{\eta^{1/2}}.
\]
Since $\epsilon \leq 2 e^{B_2/2}(b-a)^{1/2}$, we have $\delta \leq (b-a)$, so 
\begin{align*}
\log N &\leq \log \biggl(12\Bigl(\frac{b-a}{\delta}\Bigr)^2\biggr) + C (B_2-B_1)^{1/2}\frac{e^{B_2/4}(b - a)^{1/4}}{\eta^{1/2}} \\
&\leq 8 \log \biggl(192^{1/8} \frac{e^{B_2/4}(b-a)^{1/4}}{\epsilon^{1/2}}\biggr) + C (B_2-B_1)^{1/2}\frac{e^{B_2/4}(b - a)^{1/4}}{\epsilon^{1/2}} \\
&\leq C\{1+(B_2-B_1)^{1/2}\}\frac{e^{B_2/4}(b - a)^{1/4}}{\epsilon^{1/2}},
\end{align*}
as required.  Now suppose that $\epsilon > 2 e^{B_2/2}(b-a)^{1/2}$, and consider the single bracket $f_L(x) := 0$ and $f_U(x) := e^{B_2}$ for all $x \in [a,b]$. Then $d_{\mathrm{H}}^2(f_U,f_L) = e^{B_2}(b-a) \leq \epsilon^2$, and it follows that $H_{[]}(\epsilon, \F''([a, b], B_1, B_2), \shel,  [a, b]) = 0$ in this case. This proves the result for the subclass $\F''([a, b], B_1,B_2)$.

Now consider an arbitrary function in $\F'([a, b], B_1,B_2)$ of the form $f \one_{\{B_1 \leq \log f \le B_2\}}$ for some non-negative $f$ for which $\log f: [a, b] \rightarrow [-\infty, \infty)$ is concave.  Then $\{x:B_1 \leq \log f(x) \leq B_2\}$ can be written as a union of at most two intervals. It therefore follows that
\begin{equation*}
f \one_{\{B_1 \leq \log f \le B_2\}} \in \F''([a, b], B_1, B_2) + \F''([a, b], B_1, B_2). 
\end{equation*}
Hence, by Lemma~\ref{golo}, 
\begin{equation*}
H_{[]}\bigl(\epsilon, \F'([a, b], B_1, B_2), \shel, [a, b]\bigr) \leq 2 H_{[]}\bigl(\epsilon/2^{1/2},\F''([a, b], B_1, B_2), \shel, [a, b]\bigr),
\end{equation*}
so the result follows.  
\end{proof}

\subsubsection{Auxiliary result for the proof of Theorem~\ref{k1}}
\label{Sec:vdg}

The following is the key empirical processes result used in the proof of Theorem~\ref{k1}.  
\begin{theorem}[\citet{vandegeerbook}, Corollary~7.5]
\label{imee}
Let $f_0 \in \F$ and let $\F(f_0, \delta)  := \bigl\{f \in \F: f \ll f_0, d_{\mathrm{H}}(f, f_0) \leq \delta \bigr\}$. Suppose $\Psi: (0, \infty) \rightarrow (0, \infty)$ is a function such that     
\begin{equation*}
  \Psi(\delta) \geq \max \biggl\{\delta, \int_0^{\delta} H_{[]}^{1/2}\bigl(2^{1/2} \epsilon, \F(f_0, 2\delta), \shel\bigr) \, d\epsilon\biggr\} \qt{for every $\delta > 0$} 
\end{equation*}
and such that $\delta \mapsto \delta^{-2}\Psi(\delta)$ is decreasing on $(0,\infty)$.  Let $\hat{f}_n$ denote the maximum likelihood estimator over $\mathcal{F}$ based on $X_1,\ldots,X_n \stackrel{\mathrm{iid}}{\sim} f_0$.  If $\delta_* > 0$ is such that $n^{1/2} \delta_*^2 \geq C \Psi(\delta_*)$, then for every $\delta \geq \delta_*$,
\begin{equation*}
  \P \bigl\{ \kld(\hat f_n , f_0) > \delta^2 \bigr\} \leq C \exp \biggl(\frac{-n \delta^2}{C^2} \biggr). 
\end{equation*}
\end{theorem}
In fact, \citet[][Corollary~7.5]{vandegeerbook} relies on a bracketing entropy upper bound in Hellinger distance for $\bar{\mathcal{F}}(f_0,\delta) := \bigl\{\frac{f+f_0}{2}: f \in \mathcal{F}, f \ll f_0, d_{\mathrm{H}}\bigl(\frac{f+f_0}{2}, f_0\bigr) \leq \delta \bigr\}$, where the restriction $f \ll f_0$ can be included because the support of $\hat{f}_n$ is contained in the support of $f_0$.  But for any non-negative functions $f_0$, $f_L$ and $f_U$ with $f_L \leq f_U$, we have
\[
\Bigl(\frac{f_U+f_0}{2}\Bigr)^{1/2} - \Bigl(\frac{f_L+f_0}{2}\Bigr)^{1/2} \leq \frac{1}{2^{1/2}}(f_U^{1/2} - f_L^{1/2}).
\]
Moreover, since the Hellinger distance is jointly convex in its arguments, if $d_{\mathrm{H}}\bigl(\frac{f+f_0}{2},f_0\bigr) \leq \delta$, then
\[
d_{\mathrm{H}}(f,f_0) \leq 2\biggl\{d_{\mathrm{H}}\biggl(f,\frac{f+f_0}{2}\biggr) + d_{\mathrm{H}}\biggl(\frac{f+f_0}{2},f_0\biggr)\biggr\} - d_{\mathrm{H}}(f,f_0) \leq 2\delta,
\]
so $H_{[]}\bigl(2^{1/2} \epsilon, \bar{\mathcal{F}}(f_0, \delta), \shel\bigr) \leq H_{[]}\bigl(2^{1/2} \epsilon, \F(f_0, 2\delta), \shel\bigr)$.

\end{document}